\documentclass[11pt]{amsart}

\usepackage{amsmath}
\usepackage[backref=page]{hyperref}
\usepackage{tikz,float,caption}
\usetikzlibrary{arrows.meta,calc,decorations.markings,patterns,cd,patterns.meta}

\usepackage[bb=ams,scr=euler]{mathalpha}

\usepackage{setspace}
\usepackage{enumitem}
\setenumerate{
  wide=10pt,
  leftmargin=0pt,
  labelwidth=0pt,
  label=(\roman*),
  topsep=0pt,
  listparindent=0pt,
  parsep=4pt}

\usepackage{thmtools}
\declaretheoremstyle[headfont=\normalsize\normalfont\bfseries,notefont=\mdseries, notebraces={(}{)},bodyfont=\normalfont,postheadspace=0.5em]{basicstyle}
\declaretheoremstyle[headfont=\normalsize\normalfont\bfseries,notefont=\mdseries,
notebraces={(}{)},bodyfont=\normalfont\itshape,postheadspace=0.5em]{italstyle}

\declaretheorem[style=italstyle,name=Theorem,numberwithin=section]{theorem}

\declaretheorem[style=italstyle,name=Proposition,sibling=theorem]{prop}
\declaretheorem[style=italstyle,name=Lemma,sibling=theorem]{lemma}
\declaretheorem[style=basicstyle,name=Proof,numbered=no,qed=$\square$]{preproof}
\renewenvironment{proof}{\preproof}{\endpreproof}

\makeatletter
\newcommand{\abs}[1]{\left|#1\right|}
\newcommand{\bd}{\partial}

\newcommand{\C}{\mathbb{C}}
\newcommand{\dist}{\mathrm{dist}}
\renewcommand{\d}{\mathrm{d}}
\newcommand{\id}{\mathrm{id}}
\newcommand{\norm}[1]{\left\lVert#1\right\rVert}
\newcommand{\pd}[2]{\frac{\partial #1}{\partial #2}}
\newcommand{\R}{\mathbb{R}}
\renewcommand\section{\@startsection{section}{1}{0pt}{-3.5ex \@plus -1ex \@minus -.2ex}{2.3ex \@plus.2ex}{\centering\itshape}}
\newcommand{\set}[1]{\left\{#1\right\}}
\renewcommand{\subsection}{\@startsection{subsection}{2}\z@{.5\linespacing\@plus.7\linespacing}{-.5em}{\normalfont\itshape}}
\newcommand{\Z}{\mathbb{Z}}

\makeatother

\title[A Symplectic Cohomology Barcode For Contactomorphisms]{Shelukhin's Hofer distance and a symplectic cohomology barcode for contactomorphisms}
\author{Dylan Cant}
\begin{document}
\maketitle

\begin{abstract}
  This paper constructs a persistence module of Floer cohomology groups associated to a contact isotopy of the ideal boundary of a Liouville manifold. The barcode (or, bottleneck) distance between the persistence modules is bounded from above by Shelukhin's Hofer distance. Moreover, the barcode is supported (i.e., has spectrum) on the lengths of translated points of the contactomorphism. We use this structure to prove various existence results for translated points and to construct spectral invariants for contactomorphisms which are monotone with respect to positive paths and continuous with respect to Shelukhin's Hofer distance. While this paper was nearing completion, the author was made aware of similar upcoming work by Djordjevi\'c, Uljarevi\'c, Zhang.
\end{abstract}

\section{Introduction}
\label{sec:introduction}
Let $(W,\lambda)$ be a Liouville manifold and let $Y$ denote its ideal contact boundary. The main results in this paper concern the relationship between a persistence module of symplectic cohomology groups associated to a contactomorphism $\varphi$ of $Y$, the translated points of $\varphi$, and Shelukhin's Hofer distance for contactomorphisms from \cite{shelukhin_contactomorphism}.

Persistence modules were introduced into symplectic topology in \cite{polterovich_shelukhin_persistence_1}; the reader is referred to \cite{chazal_et_al_structure_stability_pmod_book,usher_zhang,vukasin_zhang,persistence_book,kislev_shelukhin,shelukhin_viterbo,shelukhin_zoll,cineli_ginzburg_gurel_entropy,ginzburg_gurel_mazzucchelli_entropy,cant_oscillation_energy,fender-lee-sohn-entropy} for more details on persistence modules, barcodes, and their role in symplectic topology.

Symplectic cohomology is a well-studied invariant of Liouville domains built out of the Floer cohomology groups of certain Hamiltonian systems, see \cite{viterbo_functors_and_computations_1,seidel-biased,ritter_TQFT,uljarevic_floer_homology_domains,merry_ulja,vukasin_zhang,uljarevic_ssh,ulja_zhang,ulja_drobnjak,shelukhin_viterbo,shelukhin_zoll,PA_spectral_diameter,fender-lee-sohn-entropy}.

Translated points are a generalization of fixed points for contactomorphisms defined in \cite{sandon_12,sandon_13} (as a special kind of leafwise intersection point); see \cite{albers_merry,shelukhin_contactomorphism, AFM15,meiwes_naef,oh_legendrian_entanglement,oh_shelukhin_anti_ci,allais_zoll,cant_sandon_conj} for research on translated points.

It is well-known that symplectic cohomology is closely related to \emph{Rabinowitz Floer homology} (RFH), see \cite{cieliebak_frauenfelder,cieliebak_frauenfelder_oancea} and in particular \cite{albers_frauenfelder_nonlinear_maslov,albers_merry,albers_merry_orderability_non_squeezing} which associate algebraic invariants to contactomorphisms using RFH. Related invariants for Legendrians involving barcodes and the Shelukhin-Hofer norm can be found in \cite{rizell_sullivan_persistence_1,rizell_sullivan_augmentation}. The work of \cite{oh_legendrian_entanglement,oh_shelukhin_anti_ci,oh_yu_1,oh_yu_2} defines algebraic invariants for contactomorphisms and Legendrians using the framework of \emph{contact instantons}. Work in progress \cite{albers_shelukhin_zapolsky} defines similar invariants for prequantization spaces (without requiring a Liouville filling!), using the technology of Lagrangian Floer cohomology.

The structures introduced in this paper are fairly close to those in \cite{djordjevic_uljarevic_zhang}; however, the results were reached independently and are different in certain ways.

\subsection{Definition of terms and statement of results}
\label{sec:definition-terms}

\subsubsection{Group of contactomorphisms}
\label{sec:group-contactomorphisms}

Let $\Gamma$ be the identity component in the group of contactomorphisms of $Y$. We will frequently consider families $\varphi_{t}\in \Gamma$, $t\in [0,1]$, satisfying $\varphi_{0}=1$, and we implicitly require the extension to $\R$ given by $\varphi_{t+1}=\varphi_{t}\varphi_{1}$ to be smooth. We refer to such a family as a \emph{system} or as a \emph{contact isotopy}. The space of such systems up to homotopy is a model for the universal cover of $\Gamma$.

\subsubsection{Discriminant points, Reeb flow, and translated points}
\label{sec:discriminant-points}

A \emph{discriminant point} of $\varphi\in \Gamma$ is a point $y\in Y$ so that $\varphi(y)=y$ and $(\varphi^{*}\alpha)_{y}=\alpha_{y}$ for any contact form $\alpha$; see \cite{givental_quasimorphism,chekanovQF,albers_frauenfelder_nonlinear_maslov,colin_sandon,gkps}. We denote by $\Gamma^{\times}\subset \Gamma$ the subset of contactomorphisms without discriminant points. The complement $\Gamma\setminus \Gamma^{\times}$ is called the \emph{discriminant locus}.

Given a contact form $\alpha$ on $Y$ let $R_{s}=R^{\alpha}_{s}$ denote the corresponding Reeb flow by time $s$. For $\varphi\in \Gamma$, introduce the \emph{spectrum} $\mathrm{Spec}_{\alpha}(\varphi)$ as the set of numbers $s\in \R$ so that $\varphi^{-1}\circ R_{s}$ has a discriminant point.

A \emph{translated point of length $s$} for $(\varphi,\alpha)$ is a discriminant point for $\varphi^{-1}\circ R_{s}$. Thus the spectrum is the set of lengths of translated points. It should be noted that our convention is that a translated point is a pair $(s,y)\in \R\times Y$. This convention differs slightly from, e.g., \cite{shelukhin_contactomorphism, merry_ulja}. If two translated points $(s,y)$ and $(s',y')$ have $y'\ne y$, they are said to be \emph{geometrically distinct}.

\subsubsection{Persistence module associated to a contactomorphism}
\label{sec:persistence-module-intro}

As in \cite{uljarevic_floer_homology_domains,merry_ulja,ulja_zhang,uljarevic_ssh,ulja_drobnjak}, to every system $\varphi_{t}$ with $\varphi_{1}\in \Gamma^{\times}$ there is an associated Floer cohomology group $\mathrm{HF}(\varphi_{t})$, defined as the Floer cohomology of any Hamiltonian system on $W$ whose ideal restriction agrees with $\varphi_{t}$; we give the construction in \S\ref{sec:sympl-cohom-pers}.

For each $s\not\in \mathrm{Spec}_{\alpha}(\varphi)$ let: $$V_{\alpha,s}(\varphi):=\mathrm{HF}(\varphi_{t}^{-1}\circ R_{st}),$$ and, for $s\le s'$, define continuation maps $\mathfrak{c}_{s,s'}:V_{\alpha,s}(\varphi)\to V_{\alpha,s'}(\varphi)$ by counting continuation cylinders, similarly to, e.g., \cite{ulja_zhang}; the precise details are given in \S\ref{sec:cont-cylind}. The data $(V_{\alpha}(\varphi),\mathfrak{c})$ forms a \emph{persistence module}, and hence has an associated \emph{barcode} $\mathfrak{B}_{\alpha}(\varphi)$. Lemma \ref{lemma:endpoint-spectrum} below implies the endpoints of bars lie in $\mathrm{Spec}_{\alpha}(\varphi)$. The persistence module depends only on the image $(\varphi_{1},[\varphi_{t}])$ in the universal cover, up to isomorphism.

In the case $\varphi_{t}=\id$, the barcode $\mathfrak{B}_{\alpha}(\id)$ is a known object; indeed, the boundary depth of the unit element appears in, e.g., \cite{benedetti-kang}; see also \cite[\S3]{vukasin_zhang} and \cite{fender-lee-sohn-entropy} for a related barcode for Reeb flows is defined using action values as the persistence parameter.

The colimit of $V_{\alpha,s}(\varphi)$ as $s\to\infty$ is independent of both $\varphi$ and $\alpha$ and is called the \emph{symplectic cohomology} of $W$. See \cite{seidel-biased,ritter_TQFT,PA_spectral_diameter} for a nice survey of this invariant.

\subsubsection{Bounding the barcode distance in terms of Shelukhin's Hofer distance}
\label{sec:stat-main-result}
Our first main result is:
\begin{theorem}\label{theorem:main}
  The barcode distance between $\mathfrak{B}_{\alpha}(\varphi_{1})$ and $\mathfrak{B}_{\alpha}(\varphi_{0})$ is bounded from above by Shelukhin's Hofer distance $\mathrm{dist}_{\alpha}(\varphi_{0},\varphi_{1})$, where the distance is measured in the universal cover.
\end{theorem}
The key step in the proof is to construct an \emph{interleaving} between the persistence modules $V_{\alpha}(\varphi_{0})$ and $V_{\alpha}(\varphi_{1})$ and then apply the famous isometry theorem relating the interleaving distance with the barcode distance; see \S\ref{sec:barc-dist-funct}, \S\ref{sec:interleavings-isometry}.

References on this isometry theorem from the perspective of topological data analysis are \cite{cohen_steiner_edelsbrunner_harer, chazal_cohen_steiner_glisse_guibas_oudot_2009,bauer_lesnick,chazal_et_al_structure_stability_pmod_book}.

Shelukhin's Hofer distance and the barcode distance function are recalled in \S\ref{sec:shel-hofer-dist}, \S\ref{sec:barc-dist-funct}, respectively. The proof of Theorem \ref{theorem:main} is completed in \S\ref{sec:interl-two-pers}.

\subsubsection{Endpoints of bars lie in the spectrum}
\label{sec:bound-numb-transl}
An important technical lemma for our applications is:
\begin{lemma}\label{lemma:endpoint-spectrum}
  Suppose that $[s_{0},s_{1}]$ is disjoint from $\mathrm{Spec}_{\alpha}(\varphi)$. Then no bars in $\mathfrak{B}_{\alpha}(\varphi)$ have endpoints in $[s_{0},s_{1}]$.
\end{lemma}
This follows from \cite{ulja_zhang}; the argument is briefly recalled in \S\ref{sec:proof-lemma-endpoint-spectrum}.

\subsection{Consequences of the main results}
\label{sec:cons-main-results}

We digress for a moment on some applications of our main theorem.

\subsubsection{A vanishing result for full symplectic cohomology}
\label{sec:vanish-result-full}

If $\varphi$ has no translated points for some contact form $\alpha$, then Theorem \ref{theorem:main} implies $\mathfrak{B}_{\alpha}(\varphi)$ is empty. Indeed, in this case, the spectrum $\mathrm{Spec}_{\alpha}(\varphi)$ is empty and therefore there can be no endpoints of bars. Consequently, $\mathfrak{B}_{\alpha}(\id)$ has only finite bars or some number of bars $(-\infty,\infty)$.

Following ideas of \cite{ritter_TQFT}, we explain in \S\ref{sec:no-half-infinite-bars} why the absence of half-infinite bars in $\mathfrak{B}_{\alpha}(\id)$ implies the full symplectic cohomology vanishes (and hence there are no infinite bars at all); briefly, the reason is that the unit for the pair-of-pants product always represents a bar starting at $s=0$, and this bar is finite if and only if $\mathrm{SH}^{*}(W)$ is zero. Consequently, Theorem \ref{theorem:main} implies that the full symplectic cohomology vanishes whenever $Y$ has a contactomorphism without translated points, i.e.,
\begin{prop}
  If $Y$ is the ideal boundary of $W$ and the full symplectic cohomology of $W$ is non-zero, then every $\varphi\in \Gamma$ has at least one translated point.
\end{prop}

In \cite{merry_ulja} it is shown that if $\varphi$ has no translated points then the full symplectic cohomology is finite-dimensional. It should be noted that there is no known example of a Liouville manifold $W$ with finite-dimensional non-zero symplectic cohomology.

This vanishing result for full symplectic cohomology appears to be deduceable by combining a result from \cite{albers_merry}, which states that \emph{Rabinowitz Floer homology} (RFH) vanishes whenever $Y$ has a contactomorphism without translated points, with a result from \cite{ritter_TQFT}, which states that RFH vanishes if and only if the full symplectic cohomology vanishes.

\subsubsection{Estimating the size of the spectrum}
\label{sec:cons-theorem-main}
The main observation is that the number of distinct values attained as endpoints of bars in $\mathfrak{B}_{\alpha}(\varphi)$ bounds from below the size of the spectrum and hence the number of translated points. We demonstrate the argument with an example.

As explained in \S\ref{sec:vanish-result-full}, $\mathfrak{B}_{\alpha}(\id)$ is never empty: there is always a bar $(0,\gamma)$ where $\gamma$ is the period of a Reeb orbit of $\alpha$ or $\gamma$ is infinite. Thus, if the barcode distance between $\mathfrak{B}_{\alpha}(\varphi)$ and $\mathfrak{B}_{\alpha}(\id)$ is less than $\gamma/2$, then $\mathfrak{B}_{\alpha}(\varphi)$ has at least one non-empty bar. In particular, we recover a result from \cite{shelukhin_contactomorphism} which states that:
\begin{equation*}
  \begin{aligned}
    2\norm{\varphi}_{\alpha}&<\text{minimal positive action of an $\alpha$-Reeb orbit}\\
                            &\hspace{2cm}\implies \varphi\text{ has a translated point}.
  \end{aligned}
\end{equation*}
The existence of translated points for $\varphi$ is unchanged if $\varphi$ is replaced by $R_{s}\circ \varphi$. Thus one can replace $2\norm{\varphi}_{\alpha}$ by $2\inf_{s}\norm{R_{s}\circ \varphi}_{\alpha}$; this latter quantity is called the \emph{$\alpha$-oscillation energy} of $\varphi$, see \cite{shelukhin_contactomorphism}; see also \cite{oh_shelukhin_anti_ci}.

Generalizing the argument yields:
\begin{prop}\label{prop:cons-1}
  Let $\mathscr{E}_{\delta}\subset \R$ be the set of endpoints of bars in $\mathfrak{B}_{\alpha}(\id)$ of length $\delta$ or more, and let $K_{\delta}(\mathscr{E})$ be the minimal number of points in a subset $\mathscr{F}\subset \R$ so that the $\delta/2$-neighborhood around $\mathscr{F}$ contains $\mathscr{E}_{\delta}$. Then $2\inf_{s}\norm{R_{s}\circ \varphi}_{\alpha}<\delta$ implies $\varphi$ has at least $K_{\delta}(\mathscr{E})$ many lengths of translated points for $\alpha$.\hfill$\square$
\end{prop}

\subsubsection{Existence of two translated points when symplectic cohomology vanishes}
\label{sec:exist-two-transl}

In this section, suppose the full symplectic cohomology vanishes, and the $\alpha$-oscillation energy of $\varphi$ is less than the minimal action of a Reeb orbit of $\alpha$, as above. Then $\mathfrak{B}_{\alpha}(\id)$ must have a \emph{finite} bar $(0,\gamma)$, where $\gamma$ is a positive action of a Reeb orbit. The assumption on the $\alpha$-oscillation energy implies $\mathfrak{B}_{\alpha}(R_{s}\circ \varphi)$ also has a finite bar (for some $s$), and hence $\varphi$ has two translated points, with two different lengths.

\subsubsection{Existence of infinitely many translated points on ellipsoids}
\label{sec:exist-infin-many}

Consider the Hamiltonian isotopy of $W=\C^{n}$ generated by:
\begin{equation*}
  r_{a}(z)=a_{1}^{-1}\pi\abs{z_{1}}^{2}+\dots+a_{n}^{-1}\pi\abs{z_{n}}^{2},
\end{equation*}
where $a_{1}\le \dots \le a_{n}$. It is well-known that $X_{r_{a}}$ has ideal restriction equal to the Reeb flow associated to the ellipsoid $\set{r_{a}=1}$. Let $\alpha$ be the induced contact form on the ideal boundary. Then $\mathrm{Spec}_{\alpha}(\id)$ equals $a_{1}\Z\cup \dots \cup a_{n}\Z$. Moreover, when $s$ is not in the spectrum, the Hamiltonian Floer cohomology $V_{\alpha,s}(\id)=\mathrm{HF}(R_{s})$ has a single generator whose Conley-Zehnder index equals:
\begin{equation*}
  \mathrm{CZ}(s)=n+2\textstyle\sum_{j=1}^{n} \lfloor s/a_{j}\rfloor.
\end{equation*}
This can be proved by using the autonomous system $H=sr_{a}$, which has a single non-degenerate orbit at $z=0$, to compute $\mathrm{HF}(R_{s})$. See \cite{cant_thesis} for the computation of the Conley-Zehnder index.

Since continuation maps preserve Conley-Zehnder indices, \emph{every interval} in $\R\setminus(a_{1}\Z\cup \dots \cup a_{n}\Z)$ is a bar in $\mathfrak{B}_{\alpha}(\id)$. Proposition \ref{prop:cons-1} then implies:
\begin{prop}
  Let $0<a_{1}\le \dots \le a_{n}$. If $\varphi$ is a contactomorphism of the standard contact sphere whose oscillation energy is less than $a_{1}$, and $\alpha_{a}$ is the contact form associated to the ellipsoid $\set{r_{a}=1}$, as above, then $\varphi$ has infinitely many $\alpha_{a}$-translated points, with lengths diverging to infinity.
\end{prop}
\begin{proof}
  It suffices to prove that, for any $\epsilon>0$, the complement: $$\R\setminus (a_{1}\Z\cup \dots \cup a_{n}\Z)$$ contains infinitely many intervals of length greater than $a_{1}-\epsilon$. Without loss of generality, suppose $a_{1}=1$. Consider the torus $T=\R/a_{2}\Z\times \dots\times \R/a_{n}\Z$, and consider the projection of the diagonal sequence $(k,\dots,k)$, $k\in \mathbb{N}$, to a sequence $x_{k}\in T$. Since $x_{k}$ is the $k$th iterate of $(0,\dots,0)$ under a translation isometry, Poincar\'e recurrence implies $x_{k}$ eventually enters the $\epsilon/2$ cube around $(0,\dots,0)$.\footnote{The image of the $\epsilon/4$ cube eventually intersects the $\epsilon/4$ cube around $0$, and since we are iterating a translation $x$ must then lie in the $\epsilon/2$ cube around $0$.} For such $k$, we have $\abs{k-d_{j}a_{j}}< \epsilon/2$ for some positive integers $d_{j}$. In particular, the interval $[k-1+\epsilon/2,k-\epsilon/2]$ does not contain any multiples of $a_{2},\dots,a_{n}$ (since each $a_{j}\ge 1$). Since $k$ can be taken to be arbitrarily large in the conclusion of Poincar\'e recurrence, the desired result follows.
\end{proof}

\subsubsection{Spectral invariants}
\label{sec:spectral-invariants}

To a non-zero element $\mathfrak{e}\in \mathrm{SH}^{*}(W)$ one can associate a \emph{spectral invariant} for $\varphi_{t}$ as the infimal number $s$ so that $\mathfrak{e}$ lies in the image of $$\mathrm{HF}(\varphi_{t}^{-1}\circ R_{st})\to \mathrm{SH}(W).$$ Let us denote this number by $c_{\alpha}(\mathfrak{e},\varphi_{t})$. Since the bars of the form $(a,\infty)$ in $\mathfrak{B}_{\alpha}(\varphi)$ are in bijection with a basis for the full symplectic cohomology $\mathrm{SH}^{*}(W)$, the spectral invariant is always the left-endpoint of some bar $(a,\infty)$. The spectral invariant depends only on the projection $(\varphi_{1},[\varphi_{t}])$ to the universal cover.

To obtain a real-valued measurement, one should take $\mathfrak{e}$ to be non-zero in the quotient $\mathrm{SH}(W)/\Pi$ where $\Pi$ is the span of the basis elements corresponding to fully infinite bars. As an example, one can pick $\mathfrak{e}$ to be the basis element in $\mathrm{SH}(W)$ corresponding to a positive half-infinite bar.

The methods in this paper yield the following:
\begin{prop}\label{prop:lipschitz-spectral}
  Given $\varphi_{0,t},\varphi_{1,t}$, one has: $$\abs{c_{\alpha}(\mathfrak{e},\varphi_{0,t})-c_{\alpha}(\mathfrak{e},\varphi_{1,t})}\le \dist_{\alpha}((\varphi_{0,1},[\varphi_{0,t}]),(\varphi_{1,1},[\varphi_{1,t}])),$$ where the distance is measured in the universal cover. Moreover, $c_{\alpha}(\mathfrak{e},\varphi_{0,t})$ is continuous with respect to $\alpha$.
\end{prop}
\begin{prop}\label{prop:mono-spectral}
  If $\varphi_{s,t}$ is a path so that $\varphi_{s,0}=\id$ and $\varphi_{s,1}$ is positive, and $\mathfrak{e}$ is a non-zero element of the quotient $\mathrm{SH}(W)/\Pi$, then $c_{\alpha}(\mathfrak{e},\varphi_{s,t})$ is strictly increasing function of $s$.
\end{prop}

The proofs are given in \S\ref{sec:lipsh-cont-spectr-proof} and \S\ref{sec:mono-spectr-proof}. It is interesting to compare these spectral invariants with other spectral invariants for contactomorphisms appearing in the literature, e.g., \cite{sandon_non_squeezing,albers_shelukhin_zapolsky,albers_merry_orderability_non_squeezing,oh_legendrian_entanglement,oh_shelukhin_anti_ci,oh_yu_1,oh_yu_2}.

\subsubsection{Orderability of ideal boundaries}
\label{sec:order-ideal-bound}
A direct consequence of Lemma \ref{sec:mono-spectr-proof} is that the ideal boundary of a Liouville manifold with non-vanishing symplectic cohomology is \emph{orderable}:
\begin{prop}
  If $Y$ is the ideal boundary of a Liouville manifold $W$ satisfying $\mathrm{SH}(W)\ne 0$, then there can be no positive contractible loop of contactomorphisms, i.e., $Y$ is orderable in the sense of \cite{ep2000}.
\end{prop}
\begin{proof}
  Let $\mathfrak{e}$ be the non-zero unit, and consider $c(t)=c_{\alpha}(\mathfrak{e},\varphi_{s,t})$ where $\varphi_{s,1}$ is a positive contractible loop. Then, because the spectral invariant depends only on the projection to the universal cover, $c(1)=c(0)$. However, Proposition \ref{prop:mono-spectral} implies $c(1)>c(0)$. This completes the proof.
\end{proof}

This result has already been proved via a different method in \cite{chantraine_colin_d_rizell}, and is closely related to the work of \cite{albers_merry_orderability_non_squeezing} which establishes a similar result in the context of RFH, and \cite{merry_ulja}, which prove the ideal boundary of a Liouville domain with infinite dimensional $\mathrm{SH}$ is orderable. The relation between positive loops of contactomorphisms and symplectic cohomology groups is studied further in \cite{cant_hedicke_kilgore}.

\subsubsection{Non-orderable ideal boundaries and the boundary depth}
\label{sec:non-orderable-ideal}

As explained by Shelukhin in a private correspondence, in the case when $Y$ is non-orderable, one can use the barcode $\mathfrak{B}_{\alpha}(\id)$ to bound the minimal ``size'' of any homotopy between a contractible positive loop and the constant loop, similarly to the result \cite[Theorem 1.11]{ekp}.

Let $\zeta_{s,t}$, $s\in \R/\Z$ and $t\in [0,1]$, be a homotopy of based loops in $\Gamma$ so that $\zeta_{s,1}$ is positive and $\zeta_{s,0}$ is the constant loop at the identity. Let $h_{s,t}$ be the contact Hamiltonian generating the loop $s\mapsto \zeta_{s,t}$; see \S\ref{sec:shel-hofer-dist} for the definition of $h_{s,t}$. Following \cite[pp.\,1641]{ekp}, one defines the number:
\begin{equation*}
  \mu_{\alpha}(\zeta_{s,t})=-\min_{s,t,y}h_{s,t}(y).
\end{equation*}
Shelukhin explained that our methods imply the following estimate:
\begin{prop}\label{prop:boundary-depth-estimate}
  Pick any system $\varphi_{t}$, and let $b$ be the length of the longest bar appearing in $\mathfrak{B}_{\alpha}(\varphi)$ (i.e., the boundary depth). Let $\zeta_{s,t}$ be a null-homotopy of a positive loop $s\mapsto \zeta_{s,1}$, as above. Then:
  \begin{equation*}
    b \le \mu_{\alpha}(\zeta_{s,t}).
  \end{equation*}
  In particular, the boundary depth of $\mathfrak{B}_{\alpha}(\varphi)$ is uniformly bounded whenever $Y$ is non-orderable in the sense of \cite{ep2000}.
\end{prop}
The proof is given in \S\ref{sec:bound-bound-depth}. See \cite{AFM15,albers_merry_orderability_non_squeezing} for related results in the context of RFH.

\subsection{Additional definitions}
\label{sec:addit-defin}

In this section we review some of the concepts used in the preceding results.

\subsubsection{Shelukhin's Hofer distance for contactomorphisms}
\label{sec:shel-hofer-dist}

Let $\varphi_{s}:Y\to Y$, $s\in [0,1]$ be a contact isotopy, and define $h_{s}(\varphi_{s}(x)):=\alpha_{\varphi_{s}(x)}(\varphi_{s}'(x))$ where $\alpha$ is some choice of a contact form. In \cite{shelukhin_contactomorphism} the Hofer-type (pseudo)-norm is considered:
\begin{equation*}
  \norm{\varphi}_{\alpha}:=\inf_{\varphi_{s}}\int_{0}^{1}\max_{y} \abs{h_{s}(y)}\d t.
\end{equation*}
The infimum is taken over isotopies $\varphi_{s}$ which represent $\varphi$ in the universal cover. The associated (pseudo)-distance function is:
\begin{equation*}
  \dist_{\alpha}(\varphi_{0},\varphi_{1})=\norm{\varphi_{1}\varphi_{0}^{-1}}_{\alpha}.
\end{equation*}
Concretely, if $\varphi_{0,t}$ and $\varphi_{1,t}$ represent elements in the universal cover, one considers all extensions $\varphi_{s,t}$ so that $\varphi_{s,0}=\id$. Each such element induces a path from $\varphi_{s,1}$ joining the time $1$ maps. Differentiating the path $\varphi_{s,1}$ with respect to $s$ produces a contact Hamiltonian $h_{s}:Y\to \R$, and the distance is the minimum of $\int \max_{y}\abs{h_{s}(y)}\d t$ over all such choices $\varphi_{s,t}$.

Note that \cite[Theorem A]{shelukhin_contactomorphism} shows the pseudo-norm is non-zero on every element whose time-one map is not a lift of the identity (i.e., if we do not work in the universal cover, then the pseudo-norm is guaranteed to be a norm).

\subsubsection{Persistence modules and their barcodes}
\label{sec:barc-dist-funct}

For this purposes of this paper, a \emph{persistence module with spectrum $\Sigma$} is a functor $(V,\mathfrak{c})$ from $(\R\setminus\Sigma,\le)$, thought of as a category where there is an arrow $s\to s'$ whenever $s\le s'$, into the category of $\Z/2$-graded finite-dimensional vector spaces over the field $\Z/2$. More prosaically, to each $s\not\in \Sigma$ one associates a vector space $V_{s}$ and to each inequality $s\le s'$ one associates a linear map $\mathfrak{c}_{s,s'}:V_{s}\to V_{s'}$, in such a way that the linear maps are functorial with respect to inequalities $s\le s'\le s''$.

For every interval $[s,s']$ disjoint from $\Sigma$, $\mathfrak{c}_{s,s'}$ is required to be an isomorphism.

A \emph{barcode} $\mathfrak{B}$ is the data of a set $X$ with a map: $$(a,b):X\to [-\infty,\infty)\times (-\infty,\infty]$$ so that $a\le b$; this is interpreted as a collection of intervals $[a(x),b(x)]$ parameterized by $x\in X$. We say that $\mathfrak{B}$ has \emph{spectrum $\Sigma$} if $a$ is valued in $\set{-\infty}\cup \Sigma$ and $b$ is valued in $\set{\infty}\cup \Sigma$.

It is shown in \cite{crawley_boevey,chazal_et_al_structure_stability_pmod_book} that persistence modules with spectrum $\Sigma$ have a \emph{normal form} described by a barcode $\mathfrak{B}$ with spectrum $\Sigma$, described as follows. For $s\not\in \Sigma$, let $X_{s}\subset X$ be those elements $x$ so that $a(x)<s<b(x)$. There is a map $T_{s}:X_{s}\to V_{s}$ whose image is a ($\Z/2$-graded) basis, and so that:
\begin{enumerate}
\item $\mathfrak{c}_{s,s'}\circ T_{s}(x)=T_{s'}(x)$ if $a(x)<s\le s'<b(x)$,
\item $\mathfrak{c}_{s,s'}\circ T_{s}(x)=0$ otherwise.
\end{enumerate}
In words, the bars containing $s$ form a basis for $V_{s}$, and the structure maps $\mathfrak{c}_{s,s'}$ respect these bases. See \cite{chazal_et_al_structure_stability_pmod_book,persistence_book} for more details.

Two barcodes $\mathfrak{B}_{1},\mathfrak{B}_{2}$ are said to be \emph{within distance} $\delta$ provided there are sets $E_{1},E_{2}$ and a bijection $S_{1}\sqcup E_{1}\simeq S_{2}\sqcup E_{2}$ so that the bijection preserves $a,b$ up to $\delta$, with the requirement that we extend $a,b$ so that $a=b$ holds on $E_{i}$. In other words, after adding some number of bars of length zero, there is a matching between the bars which does not move the endpoints more than $\delta$.

The infimal $\delta$ so that $\mathfrak{B}_{1},\mathfrak{B}_{2}$ are within $\delta$ is the \emph{barcode distance} (or \emph{bottleneck distance}) between $\mathfrak{B}_{1},\mathfrak{B}_{2}$. See \cite[\S2.2]{persistence_book}.

\subsection{Further questions}
\label{sec:further-questions}

\subsubsection{Comparison with other measurements}
\label{sec:comp-with-other}
Are there estimates on the barcode $\mathfrak{B}_{\alpha}(\varphi)$ in terms of other measurements for contactomorphisms, for instance those in \cite{ep2000,sandon_integer_metric,zapolsky_cont_groups,sandon_bi_invariant,colin_sandon,fraser_polterovich_rosen,oh_legendrian_entanglement,oh_shelukhin_anti_ci,nakamura_orderable_metric,allais_arlove}?

Another direction is to analyze the dependence of $\mathfrak{B}_{\alpha}(\varphi)$ on $\alpha$ more closely. For instance, the work of \cite{vukasin_zhang,persistence_book,usher_banach_mazur} indicates one should be able to bound the distance between logarithmic versions of $\mathfrak{B}_{\alpha}(\id),\mathfrak{B}_{\alpha'}(\id)$ in terms of the ratio between contact forms $\alpha/\alpha'$. This is related to the concept of the \emph{Banach-Mazur distance} on the space of contact forms, see \cite{rosen_zhang_CBM}.

\subsubsection{Bounded length of bars}
\label{sec:bounded-length-bars}

Which Liouville manifolds and contact forms $\alpha$ have a uniform bound on the length of the non-infinite bars in $\mathfrak{B}_{\alpha}(\varphi)$? Shelukhin's argument in \S\ref{sec:non-orderable-ideal} shows that all Liouville manifolds whose ideal boundary is non-orderable in the sense of \cite{ep2000}.

\subsubsection{Subcritical contact manifold without translated points}
\label{sec:subcritical}
In \S\ref{sec:vanish-result-full} it is shown that a contactomorphism without translated points implies the vanishing of symplectic cohomology. It is known that \emph{subcritical Weinstein manifolds} have vanishing $\mathrm{SH}$; see \cite{cieliebak_handle_chord,cieliebak_eliashberg_stein}. Does the ideal contact boundary of every subcritical Weinstein manifold have a contactomorphism without a translated point for some choice of contact form?

A slight variation, is there some subcritical $W$ so that $Y=\bd W$ has a contact form $\alpha$ so that every contactomorphism has an $\alpha$-translated point? It should be noted that the author has shown in \cite{cant_sandon_conj} that $S^{2n+1}$, for $n>1$ and with its standard contact form, has a contactomorphism without a translated point.

\subsubsection{Legendrian version}
\label{sec:legendrian-version}

The experts will recognize that much of this construction can be defined in a relative setting for a Legendrian $\Lambda_{0}$ with filling $L_{0}$ according to the dictionary:
\begin{enumerate}
\item[(contact isotopy)] Legendrian isotopy $\Lambda_{t}$ starting at $\Lambda_{0}$,
\item[(discriminant point)] intersection between $\Lambda_{1}$ and $\Lambda_{0}$ itself,
\item[(translated point)] Reeb chord from $\Lambda_{0}$ to $\Lambda_{1}$,
\item[(symplectic cohomology)] wrapped Floer cohomology of $L_{0}$.
\end{enumerate}
A version of this translation appears in \cite{cant_hedicke_kilgore}.

\subsection{Acknowledgements}
\label{sec:acknowledgements}

First and foremost, the author wishes to thank Egor Shelukhin and Octav Cornea for providing valuable guidance during the preparation of this paper. The author also wishes to thank Habib Alizadeh, Marcelo Atallah, Filip Bro\'ci\'c, Jakob Hedicke, Pierre-Alexandre Mailhot, Vuka\v{s}in Stojisavljevi\'c, and Marco Mazzucchelli for insightful discussions. Thanks as well to the authors of \cite{djordjevic_uljarevic_zhang} for providing helpful feedback and comments.

\section{Persistence module associated to a contactomorphism}
\label{sec:sympl-cohom-pers}

The outline for the rest of the paper is as follows: in \S\ref{sec:floers-equat-liouv} we recall the Floer cohomology groups for a contact-at-infinity system and a complex structure satisfying some admissibility conditions. In \S\ref{sec:cont-cylind} we define the continuation maps associated to a non-negative path. The persistence module associated to a contactomorphism is defined in \S\ref{sec:defin-pers-module}. The proof of Theorem \ref{theorem:main} is completed in \S\ref{sec:proof-1}.

\subsection{Floer cohomology in Liouville manifolds}
\label{sec:floers-equat-liouv}

The space of \emph{Floer data} $\mathfrak{A}$ is the space of tuples of (i) a contact-at-infinity Hamiltonian system $\psi_{t}$, and (ii) a time-dependent almost complex structure $J_{t}$. We suppose that $\psi_{t+1}=\psi_{t}\psi_{1}$ and that:
\begin{equation*}
  \d\psi_{1}^{-1}J_{t}(\psi_{1}(w))\d\psi_{1}=J_{t+1}(w),
\end{equation*}
i.e., $J_{t}$ is \emph{twisted periodic} with respect to the time-1 map $\psi_{1}$; see \S\ref{sec:contact-at-infinity} and \S\ref{sec:admissible-complex-structure} below.

The contact-at-infinity assumption provides an ideal restriction as a contact isotopy. Let $\mathfrak{A}^{\times}\subset \mathfrak{A}$ be the space of \emph{admissible} data, namely, those data so that:
\begin{enumerate}[label=(\alph*)]
\item the ideal restriction of $\psi_{1}$ lies in $\Gamma^{\times}$, 
\item all the fixed points of $\psi$ are non-degenerate,
\item the moduli spaces used to define the Floer differential are cut transversally.
\end{enumerate}

For each choice of admissible data, there is a \emph{Floer complex} $\mathrm{CF}(\psi_{t},J_{t})$ generated by the finitely many fixed points of $\psi_{1}$. The Floer differential counts certain solutions to Floer's equation, described in \S\ref{sec:floer-cylind-twist} below.

It is notable that the vector space $\mathrm{CF}(\psi_{t},J_{t})$ depends only on $\psi_{1}$. However, various decorations (such as action filtrations, gradings, etc) depend on the system $\psi_{t}$.

\subsubsection{Contact-at-infinity systems}
\label{sec:contact-at-infinity}

A symplectic isotopy $\psi_{t}$ is said to be \emph{contact-at-infinity} provided $\psi_{t}$ commutes with the Liouville flow outside of a compact set, and the generator $X_{t+1}=X_{t}$ is a time-dependent Hamiltonian vector field. The generating Hamiltonian functions $H_{t}$ are one-homogeneous with respect to the Liouville flow, up to the addition of a function which is constant outside of a compact set.

For the purposes of analyzing the ideal restriction, one can always require that $H_{t}$ is genuinely one-homogeneous outside of a compact set; this is only a trivial requirement when $\pi_{0}(Y)\to \pi_{0}(W)$ is injective, because one can just subtract a constant.

However, even when $\pi_{0}(Y)\to \pi_{0}(W)$ is not surjective, a generating Hamiltonian $H_{0,t}$ (which may not be truly one-homogeneous) can be deformed via $H_{s,t}:=H_{0,t}+sf,$ so that $H_{1,t}$ \emph{is} one-homogeneous outside of a compact set; one takes $f$ to be constant outside of a compact set. The induced path $\varphi_{s,t}$ has a constant ideal restriction, and so standard Floer theoretic techniques imply that the Floer cohomology of $H_{0,t}$ will be isomorphic to the Floer cohomology of $H_{1,t}$. For this reason, one can assume $H_{t}$ is one-homogeneous at infinity, without any true loss of generality. 

\subsubsection{Admissible complex structures}
\label{sec:admissible-complex-structure}

The complex structures featuring in $\mathfrak{A}$ are required to be $\omega$-tame, Liouville-equivariant in the ends, and $s$-independent. See \cite[\S2]{brocic_cant} for more details. A special class of such complex structures are those of SFT-type, see \cite{BEHWZ}.

\subsubsection{Floer differential cylinders as twisted holomorphic cylinders}
\label{sec:floer-cylind-twist}

For a choice of data $(\psi_{t},J_{t})\in \mathfrak{A}$, consider the moduli space $\mathscr{M}(\psi_{t},J_{t})$ of finite-energy ``twisted'' holomorphic maps solving:
\begin{equation*}
  \left\{
    \begin{aligned}
      &w:\C\to W\\
      &\psi_{1}(w(s,t+1))=w(s,t)\\
      &\bd_{s}w+J_{t}(w)\bd_{t}w=0.
    \end{aligned}
  \right.
\end{equation*}
Recall that we require that $J_{t+1}(w)=\d\psi_{1}^{-1}J_{t}(\psi_{1}(w))\d\psi_{1}$, i.e., $J_{t}$ is twisted-periodic with respect to $\psi_{1}$.

A similar twisted holomorphic curve equation is considered in \cite[\S3]{dostoglou_salamon}. Energy is defined in \S\ref{sec:energy-estimates} below.

Admissible data $\mathfrak{A}^{\times}\subset\mathfrak{A}$ are required to cut $\mathscr{M}(\psi_{t},J_{t})$ transversally. For such data, let $\mathscr{M}_{1}(\psi_{t},J_{t})$ be the union of the 1-dimensional components.

Note however that the moduli space $\mathscr{M}(\psi_{t},J_{t})$ is independent of the choice of system $\psi_{t}$ generating $\psi_{1}=\psi$. However, given such a choice, one can associate $u(s,t)=\psi_{t}(w(s,t))$ which solves the standard Floer equation on the cylinder, and the reason for the twisted periodic condition is so that $u$ solves a smooth PDE; see \S\ref{sec:non-negative-paths} for further discussion. 

\subsubsection{Energy estimates}
\label{sec:energy-estimates}

The energy of $w$ is the integral of $\omega$ over any strip of height $1$, i.e., $\R\times [t_{0},t_{0}+1]$; these are all the same and we typically take $t_{0}=0$. The finite energy assumption implies that $w$ converges to fixed points of $\psi_{1}$ at the $s=\pm \infty$ ends of the strip. Stokes' theorem gives the formula:
\begin{equation*}
  E(w)=\int_{w(\R,1)}\psi_{1}^{*}\lambda-\lambda=f(w(+\infty))-f(w(-\infty)).
\end{equation*}
Where $\d f=\psi_{1}^{*}\lambda-\lambda$, for $f$ which is constant outside of a compact set. Such a function $f$ exists by our assumption that the generating Hamiltonian for $\psi_{t}$ can be taken to be one-homogenous up to a constant in the ends.

Consequently, all finite energy solutions (for admissible data) automatically have an a priori bounded energy.

Fix a Riemannian metric which is translation invariant in the ends. Standard bubbling analysis shows that these energy bounds imply bounds on the first derivatives. Elliptic regularity then implies bounds on the higher derivatives; see \cite[\S3.1]{brocic_cant}.

Following \cite[\S2.2.5]{brocic_cant} and \cite{merry_ulja}, a priori energy bounds also imply a priori $C^{0}$ bounds on Floer cylinders. In \S\ref{sec:maxim-princ-cont} below we establish a maximum principle for continuation cylinders, which also proves a maximum principle for the Floer cylinders appearing in the differential.

\subsubsection{The Floer differential}
\label{sec:floer-differential}

The $\R$-action on $\mathscr{M}_{1}$ by $w(s,t)\mapsto w(s+1,t)$ is free and proper. The energy bounds in \S\ref{sec:energy-estimates} imply the quotient is compact, and hence a finite set. Define:
\begin{equation*}
  d_{\mathrm{CF}}(y):=\sum \set{u(-\infty):u\in \mathscr{M}_{1}/\R\text{ and }u(+\infty)=y}.
\end{equation*}
This is the cohomological differential on $\mathrm{CF}(\psi_{t},J_{t})$.

The homology of $\mathrm{CF}(\psi_{t},J_{t})$ is denoted $\mathrm{HF}(\psi_{t},J_{t})$; continuation maps prove its isomorphism class is independent of $J_{t}$ and depends only on the ideal restriction of $\psi_{t}$; see \S\ref{sec:cont-cylind}.

\subsubsection{Supergrading}
\label{sec:supergrading}

By definition, a \emph{supergrading} of a $\Z/2$-vector space $E$ is a decomposition $E_{0}\oplus E_{1}$.

It is well-known that for any fixed point $x$ of $\psi_{1}$, one can associate a Conley-Zehnder index $\mathrm{CZ}(\gamma)\in \Z/2$; one considers the linearization along the loop $\psi_{t}(x)$. The path of symplectic isomorphisms: $$\d \psi_{t,x}:TW_{\gamma(0)}\to TW_{\gamma(t)}$$
has a well-defined Conley-Zehnder index in $\Z/2$. This index induces a supergrading on $\mathrm{CF}(\psi_{t},J_{t})$; the index formula from \cite{floer-ham,introSFT,BEHWZ,wendl-sft,cant_thesis} implies that $d_{\mathrm{CF}}$ shifts the supergrading by $1$. The homology $\mathrm{HF}(\psi_{t},J_{t})$ therefore also inherits a supergrading.

\subsection{Continuation cylinders}
\label{sec:cont-cylind}

A path in $\mathfrak{A}$ with endpoints in $\mathfrak{A}^{\times}$ sometimes induces a \emph{continuation map} between the Floer cohomologies associated to the endpoints of the path. In the non-compact setting of Liouville manifolds, a sufficient condition for the existence of a continuation map is that the ideal restriction of the path is a \emph{non-negative path} in the group of contactomorphisms; see \S\ref{sec:non-negative-paths}.

Let us call such a path in $\mathfrak{A}$ \emph{non-negative}. One can form a diagram (a small category) whose objects are elements $(\psi_{t},J_{t})$ of $\mathfrak{A}^{\times}$ and whose morphisms $(\psi_{0,t},J_{0,t})\to (\psi_{1,t},J_{1,t})$ are homotopy classes of extension $(\psi_{s,t},J_{s,t})$ so that $\psi_{s,1}$ is non-negative, $J_{s,t}$ is twisted periodic for $\psi_{s,1}$, and where composition is given by concatenation. The continuation map construction induces a functor from this diagram to the category of supergraded $\Z/2$-modules, associating each element of $\mathfrak{A}^{\times}$ to its Floer cohomology $\mathrm{HF}(\psi_{t},J_{t})$.

\subsubsection{Definition of the Floer cohomology associated to a contact isotopy}
\label{sec:defin-floer-cohom}

For a system $\varphi_{t}\in \Gamma$ with $\varphi_{1}\in \Gamma^{\times}$, consider the small category $\Delta(\varphi_{t})$ whose objects are admissible pairs $(\psi_{t},J_{t})$, where $\psi_{t}$ is a Hamiltonian system whose ideal restriction is $\varphi_{t}$. Between any two objects we declare there to be a unique morphism.

A path in the space of data $\mathfrak{A}$ whose ideal restriction remains fixed is called \emph{compactly supported}. Clearly compactly supported paths are non-negative. Compactly supported paths can be reversed, which implies the associated continuation maps are isomorphisms.

If $\psi_{0,t},\psi_{1,t}$ both have ideal restriction $\varphi_{t}$, then there is a distinguished homotopy class of compactly supported paths from $\psi_{0,t}$ to $\psi_{1,t}$, namely, those paths $\psi_{s,t}$ whose ideal restriction remains fixed at $\varphi_{t}$. The Serre fibration property implies that (i) such a lift $\psi_{s,t}$ exists and (ii) the homotopy class of $\psi_{s,1}$, in the space of compactly supported paths from $\psi_{0,1}$ to $\psi_{1,1}$, is independent of the choice of $\psi_{s,t}$ and depends only on $\varphi_{t}$; see Figure \ref{fig:serre-fibration} and \S\ref{sec:serre-fibration-prop}. 

Thus we can define a functor from $\Delta(\varphi_{t})$ to the category of supergraded vector spaces by sending $(\psi_{t},J_{t})$ to $\mathrm{HF}(\psi_{t},J_{t})$ and each morphism is sent to the continuation map associated to the distinguished homotopy class of compactly supported paths. The limit of this functor is defined to be the Floer homology $\mathrm{HF}(\varphi_{t})$. Since the diagram has a unique morphism between any two objects, the limit map $\mathrm{HF}(\varphi_{t})\to \mathrm{HF}(\psi_{t},J_{t})$ is an isomorphism for any lift $\psi_{t}$. For details on limits of arbitrary functors, the reader is referred to \cite[\S I.5.1 and \S VIII.1.4]{aluffi}.

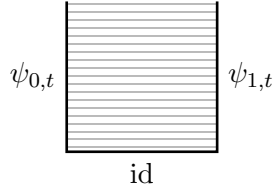
\begin{figure}[H]
  \centering
  \begin{tikzpicture}[scale=2,rotate=90]
    \path[pattern color=white!50!black,pattern={Lines[angle=45]}] (0,0) rectangle (1,1);
    \draw[line width=1pt] (1,0)--node[right](A){$\psi_{1,t}$}(0,0)--node[below](B){$\id$}(0,1)--node[left](C){$\psi_{0,t}$}(1,1);
  \end{tikzpicture}
  \caption{Defining the functor from $\Delta(\varphi_t)$ using the Serre fibration property; given $\varphi_{s,t}$, and lifts to the solid lines (i.e., $t=0$ and $s=0,1$) there is guaranteed to be some lift $\psi_{s,t}$. Evaluating at $t=1$ gives canonical homotopy class of paths from $\psi_{0,1}$ to $\psi_{0,1}$ (depending on $\varphi_{s,t}$).}
  \label{fig:serre-fibration}
\end{figure}

A final application of the Serre fibration property from \S\ref{sec:serre-fibration-prop} implies that $\mathrm{HF}(\varphi_{t})$ is a functor from the category whose objects are systems $\varphi_{t}$ with $\varphi_{1}\in \Gamma^{\times}$ and whose morphisms are homotopy classes paths $\varphi_{s,t}$ with fixed endpoints so that $\varphi_{s,1}$ is a non-negative path. Indeed, given any lifts $\psi_{0,t}$, $\psi_{1,t}$, there exists some extension to a lift $\psi_{s,t}$ whose ideal restriction is $\varphi_{s,t}$; see Figure \ref{fig:serre-fibration}. Then $\psi_{s,1}$ is a non-negative path from $\psi_{0,1}$ to $\psi_{1,1}$, inducing a continuation map $\mathrm{HF}(\psi_{0,1})\to \mathrm{HF}(\psi_{1,1})$. This continuation map is natural with respect to the choice of $\psi_{0,t},\psi_{1,t}$ and the distinguished continuation maps used to compute the limit.

In the next section \S\ref{sec:defin-pers-module}, we restrict this functor to systems of the form $\varphi_{t}^{-1}\circ R_{st}$ to define the persistence module $V_{s}(\varphi_{t})$ introduced in \S\ref{sec:persistence-module-intro}.

See \cite[\S2.2.10]{cant_hedicke_kilgore} for related discussion.

\subsubsection{Dependence on the system}
\label{sec:dependence-system}

The functoriality of $\varphi_{t}\mapsto \mathrm{HF}(\varphi_{t})$ seems to require working with the system $\varphi_{t}$ rather than its projection to the universal cover $(\varphi_{1},[\varphi_{t}])$.

As we shall show in this section, the values of this functor depend only on the projection $(\varphi_{1},[\varphi_{t}])$; however, there does not seem to be a canonical isomorphism $\mathrm{HF}(\varphi_{0,t})\to \mathrm{HF}(\varphi_{1,t})$ if $\varphi_{0,1}=\varphi_{1,1}$ and $[\varphi_{0,t}]=[\varphi_{1,t}]$.

Clearly, if $\varphi_{0,1}=\varphi_{1,1}$ and $[\varphi_{0,t}]=[\varphi_{1,t}]$, then by definition, there is some $\varphi_{s,t}$ so that $\varphi_{s,1}=\varphi_{0,1}=\varphi_{1,1}$ and $\varphi_{s,0}=\id$. By the Serre fibration property, for any lifts $\psi_{0,t}$ and $\psi_{1,t}$ to Hamiltonian systems, we can find some extension $\psi_{s,t}$ so that $\psi_{s,0}=\id$ and $\psi_{s,t}$ lifts $\varphi_{s,t}$.

In particular, the path $s\mapsto \psi_{s,1}$ is compactly supported, and hence the continuation map furnishes the desired isomorphism.

It is also important to recall the morphisms $\mathrm{HF}(\psi_{t},J_{t})\to \mathrm{SH}(W)$ is natural with respect to continuation maps associated to compactly supported path. This implies that map $\mathrm{HF}(\varphi_{t})\to \mathrm{SH}(W)$ is preserved under the above isomorphisms. This is used to show the spectral invariants $c_{\alpha}(\varphi_{t},\mathfrak{e})$ defined in \S\ref{sec:spectral-invariants} depend only on $(\varphi_{1},[\varphi_{t}])$.

\subsubsection{Topology on the space of admissible data}
\label{sec:topology-on-admissible data}

For each $(\psi_{t},J_{t})\in \mathfrak{A}$ define a basic open neighborhood to be the data of:
\begin{enumerate}
\item a compact set $\Omega\subset W$ so that $(\psi_{t},J_{t})$ is Liouville equivariant outside this compact set for each $t\in [0,1]$,
\item an open subset $V$ of $(\psi_{t},J_{t})|_{\Omega}$ in $C^{\infty}([0,1]\times \Omega)$ , and
\item an open subset $U$ of the ideal restriction of $(\psi_{t},J_{t})$ in $C^{\infty}([0,1]\times Y)$.
\end{enumerate}
The resulting open neighborhood consists of all $(\psi_{t}',J_{t}')$ so that $\psi_{t}'|_{\Omega}$ lies in $V$, is Liouville-equivariant outside of $\Omega$, and the ideal restriction lies in $U$.

These form a basis for a topology. It has the following property: if $x\mapsto \psi_{x,t}$ is a continuous map defined for $x$ in a compact space $X$, then there is a \emph{fixed} compact set $\Omega$ so that $\psi_{x,t}$ is equivariant outside $\Omega$, for all $x,t$.

Henceforth we assume all paths $\psi_{s,t}$ are continuous, satisfy $\bd_{s}\psi_{s,t}=0$ for $s$ outside a compact interval, and are infinitely differentiable in the $s$-direction.

\subsubsection{Non-negative paths and a priori energy bounds}
\label{sec:non-negative-paths}

Let $(\psi_{s,t},J_{s,t})$ be a path in $\mathfrak{A}$ so $\psi_{s,1}$ is non-negative, suppose that $\psi_{s,t},J_{s,t}$ are $s$-independent for $s$ outside of $[s_{0},s_{1}]$, and the endpoints are admissible.

By definition of $\mathfrak{A}$, we assume that $J_{s,t}$ is twisted by $\psi_{s,1}$, and hence we can form the periodic complex structure:
\begin{equation*}
  \bar{J}_{s,t}=\d\psi_{s,t}J_{s,t}\d\psi_{s,t}^{-1}, 
\end{equation*}
so that $\bar{J}_{s,t+1}=\bar{J}_{s,t}$.

When writing equations we will use $-s$ instead of $s$, since we consider the input to be the $s=+\infty$ puncture (which should correspond to the starting point of the path $\psi_{s}$).

Let $Y_{s,t},X_{s,t}$ be the generators of $\psi_{-s,t}$ with respect to $s$ and $t$. A straightforward computation shows that $w$ is $J_{-s,t}$ holomorphic if and only if the coordinate change $u(s,t)=\psi_{-s,t}(w(s,t))$ solves:
\begin{equation}\label{eq:bad-equation}
  \bd_{s}u-Y_{s,t}(u)+\bar{J}_{-s,t}(u)(\bd_{t}u-X_{s,t}(u))=0.
\end{equation}
In other words, $u$ solves an $s$-dependent Floer-type equation on the cylinder. Note that, outside of $[s_{0},s_{1}]$, $Y_{s,t}=0$ and $X_{s,t}$ equals the generators of the asymptotic systems $\psi_{s_{0},t}$ (at the right end) and $\psi_{s_{1},t}$ (at the left end).

Unfortunately, \eqref{eq:bad-equation} is not a smooth equation on the cylinder, even if $w(s,t)$ is twisted periodic, because $Y_{s,t}$ is rarely $1$-periodic in the $t$ coordinate. To obtain a smooth equation, we consider the following cut-off version for the continuation maps. First of all, assume that $\psi_{s,t}=\psi_{s,1}$ and $\psi_{s,t}=\psi_{s,0}=\id$ holds for $t$ in a neighborhood of $1$ and $0$, respectively; this is without loss of generality, since the Floer complex and differential depends only on $\psi_{s,1}$; it can be achieved by a standard time reparametrization.

Let $\mathscr{M}(\psi_{s,t},J_{s,t})$ be the moduli space of \emph{continuation cylinders} solving:
\begin{equation*}
  \left\{\begin{aligned}
    &u:\R\times \R/\Z\to W,\\
    &\bd_{s}u-\rho(t)Y_{s,t}+\bar{J}_{-s,t}(u)(\bd_{t}u-X_{s,t})=0,
  \end{aligned}\right.
\end{equation*}
where $\rho(t)$ is a cut-off function so $\rho(t)=1$ for $t\le 1-2\epsilon$ and $\rho(t)=0$ holds for $t\ge 1-\epsilon$. We suppose that $\rho'(t)\le 0$ and, whenever $\rho(t)\ne 1$, we have the equality $Y_{s,t}=Y_{s,1}$ and $X_{s,t}=0$. Note that this equation is smooth on the cylinder (since $Y_{s,t}=0$ holds for $t$ near $0$), and, moreover, agrees with the Floer differential equations for $\psi_{s_{0},t}$ and $\psi_{s_{1},t}$ for $s$ outside of $[s_{0},s_{1}]$.

The relevance of \emph{non-negativity} of $\psi_{s}$ is that it ensures an a priori energy bound:
\begin{lemma}
  There is a finite constant $C=C(\psi_{s,t},J_{s,t})$ so that:
  \begin{equation*}
    E(u)=\int \omega(\bd_{s}u-\rho(t)Y_{s,t},\bd_{t}u-X_{s,t})\d s\d t\le C,
  \end{equation*}
  holds for any $u\in \mathscr{M}(\psi_{s,t},J_{s,t})$. The bound continues to hold if one perturbs $\psi_{s,t},J_{s,t}$ in a $C^{\infty}$ small way on a compact set in $W$.
\end{lemma}
\begin{proof}
  Let $H_{s,t},K_{s,t}$ be the normalized Hamiltonian generators of $X_{s,t}$ and $Y_{s,t}$, respectively. Since $X_{s,t}$ and $Y_{s,t}$ are the vector field generators of $\psi_{-s,t}$, it follows that:
  \begin{equation}\label{eq:curvature-relation}
    -\pd{K_{s,t}}{t}+\pd{H_{s,t}}{s}+\omega(Y_{s,t},X_{s,t})=0;
  \end{equation}
  this is the well-known curvature relation for the generators of a two-parameter family of Hamiltonian diffeomorphisms; see, e.g., \cite[\S2.2.5]{cant_hedicke_kilgore} for a proof.

  We compute:
  \begin{equation*}
    E=\int u^{*}\omega+\int\rho(t)\d K_{s,t}(\bd_{t}u)-\d H_{s,t}(\bd_{s}u)+\rho(t)\omega(Y_{s,t},X_{s,t})\d s\d t.
  \end{equation*}
  Integration by parts yields:
  \begin{equation*}
    E(u)=\int u^{*}\omega-\int \rho'(t)K_{s,t}(u)+\int H_{-\infty,t}(\gamma_{-}(t))-H_{+\infty,t}(\gamma_{+}(t))\d t+\mathfrak{r},
  \end{equation*}
  where $\gamma_{\pm}$ are the asymptotic orbits and:
  \begin{equation*}
    \mathfrak{r}=-\rho(t)\pd{K_{s,t}}{t}+\pd{H_{s,t}}{s}+\rho(t)\omega(Y_{s,t},X_{s,t}).
  \end{equation*}
  By construction, $\rho(t)=1$ holds on the neighborhood where $\pd{K_{s,t}}{t}\ne 0$ and $X_{s,t}\ne 0$, and so \eqref{eq:curvature-relation} implies that $\mathfrak{r}=0$.
  
  Since $\rho'(t)\le 0$, and $K_{s,1}$ is non-positive at infinity (because it is the normalized generator for $\psi_{-s,1}$), the second term in the formula for $E$ is uniformly bounded above, by the maximum that $-\rho'(t)K_{s,1}$ achieves (which is attained on some fixed compact set). On the other hand, the remaining terms are uniformly bounded in terms of the asymptotic orbits, as desired.
\end{proof}

\subsubsection{Maximum principle for continuation cylinders}
\label{sec:maxim-princ-cont}

Suppose: $$u_{n}\in \mathscr{M}(\psi_{s,t},J_{s,t})$$ is a sequence of solutions where $\psi_{s,1}$ is a non-negative path. By \S\ref{sec:non-negative-paths}, the energy of $u_{n}$ is bounded independently of $n$.

Bubbling analysis then implies that $u_{n}$ satisfies a gradient bound, i.e., $\abs{\bd_{s}u_{n}}$ and $\abs{\bd_{t}u_{n}}$ remain bounded for any metric $g$ which is translation invariant in the symplectization end; see \cite{mcduffsalamon,brocic_cant,alizadeh-atallah-cant} for further discussion of bubbling in tame symplectic manifolds.

The maximum principle \cite[Theorem 2.4]{brocic_cant} asserts that, any sequence of solutions $u_{n}$ to the $s$-independent Floer's equation with admissible data on finite length cylinders $[a_{n},b_{n}]$ satisfying:
\begin{enumerate}
\item an energy bound $E(u_{n})\le E$,
\item a modulus bound $b_{n}-a_{n}\ge \ell$, and
\item gradient bound $\abs{\bd_{s}u}_{g}+\abs{\bd_{t}u}_{g}\le C$,
\end{enumerate}
remain in a fixed compact set $K\subset W$ depending only on $E,\ell,C$ (and the admissible data).

Apply \cite[Theorem 2.4]{brocic_cant} to these cylinders disjoint from $[s_{0},s_{1}]\times \R/\Z$ to conclude that any sequence $u_{n}$ satisfies:
\begin{equation*}
  u_{n}((-\infty,s_{0}]\times [0,1]\cup [s_{1},\infty)\times [0,1])\subset K.
\end{equation*}
However, since we have concluded the gradient bound on the entire cylinder, the central region $u_{n}([s_{0},s_{1}]\times \R/\Z)$ remains within distance $C(s_{1}-s_{0})$ from $K$. Since the metric is translation invariant in the end, it is complete, and hence the entire image $u_{n}(\R\times \R/\Z)$ remains in some fixed compact set. This completes the proof of the maximum principle for continuation cylinders.

\begin{figure}[H]
  \centering
  \begin{tikzpicture}
    \draw[dashed, every node/.style={shift={(0,-0.5)},below}] (2,0) circle (0.2 and 0.5) node{$s_{0}$} (3,0) circle (0.2 and 0.5) node{$s_{1}$};
    \draw (0,0) circle (0.2 and 0.5) coordinate (X) (5,0) circle (0.2 and 0.5) coordinate(Y);
    \path (X)+(0.2,0) arc (0:360:0.2 and 0.5) coordinate[pos=0.25] (X1) coordinate[pos=0.75](X2);
    \path (Y)+(0.2,0) arc (0:360:0.2 and 0.5) coordinate[pos=0.25] (Y1) coordinate[pos=0.75](Y2);
    \draw (X1)--(Y1) (X2)--(Y2);
  \end{tikzpicture}
  \caption{Maximum principle for continuation cylinders. The equation is translation invariant outside the interval $[s_{0},s_{1}]$.}
  \label{fig:max-principle}
\end{figure}
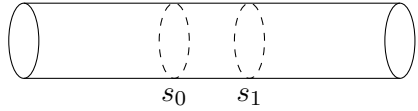

\subsubsection{Definition of the continuation map}
\label{sec:defin-cont-map}

Let $\psi_{s,t},J_{s,t}$ be a path in $\mathfrak{A}$ so that $\psi_{s,1}$ is non-negative, which we assume is $s$-independent outside of $[s_{0},s_{1}]$ and has endpoints in $\mathfrak{A}^{\times}$. One defines a continuation map by counting the rigid elements in $\mathscr{M}_{0}(\psi_{s,t},J_{s,t})$, where $\psi_{s,t}$ is perturbed to be generic on a fixed compact set so as to ensure transversality. Since $\psi_{s,t}$ is non-negative, the rigid elements in $\mathscr{M}_{0}(\psi_{s,t},J_{s,t})$ form a finite set. The continuation map $\mathfrak{c}=\mathfrak{c}(\psi_{s,t},J_{s,t}):\mathrm{CF}(\psi_{s_{0},t},J_{s_{0},t})\to \mathrm{CF}(\psi_{s_{1},t},J_{s_{1},t})$ is defined by:
\begin{equation*}
  \mathfrak{c}(x):=\sum\set{u(-\infty):u(\infty)=x\text{ and }u\in \mathscr{M}_{0}(\psi_{-s})}.
\end{equation*}
It is well-known that $\mathfrak{c}$ preserves the supergrading. As usual, examining the non-compact components in $\mathscr{M}_{1}$ proves that $\mathfrak{c}$ is a chain map with respect to the Floer differentials $d_{\mathrm{CF}}(\psi_{s_{0},1},J_{s,t})$ and $d_{\mathrm{CF}}(\psi_{s_{1},1},J_{s,t})$.

See \cite[Theorem 4]{floer-ham}, \cite[\S6]{salamon-zehnder}, \cite[pp.\ 1060]{seidel_representation} for similar definition, in the compact case.

\subsubsection{Invariance under homotopy}
\label{sec:invar-under-homot}

The argument proving that continuation maps are unchanged under deformations of the continuation data, up to chain homotopy, is well-known, see \cite[Theorem 4]{floer-ham}, \cite[Lemma 6.3]{salamon-zehnder}, and \cite[Lemma 6.13]{abouzaid_monograph}.

\subsubsection{The Serre fibration property}
\label{sec:serre-fibration-prop}
Consider the space of all Floer data and the ideal restriction map $\mathfrak{A}\to \Gamma$ as a fibration. Standard arguments involving cutting off Hamiltonians and the contractibility of spaces of almost complex structures imply this map is a Serre fibration, and, in particular, has the path lifting property. See \cite[Proposition 7]{cool_gadget}, \cite{uljarevic_floer_homology_domains,merry_ulja,ulja_drobnjak,cant_hedicke_kilgore} for related discussion.

\subsection{Definition of the persistence module}
\label{sec:defin-pers-module}

As explained in \S\ref{sec:persistence-module-intro}, the persistence module is defined as $V_{\alpha,s}(\varphi_{t})=\mathrm{HF}(\varphi_{t}^{-1}\circ R_{st})$, and the structure maps $s_{0}\le s_{1}$ is the continuation map associated to the path: $$\varphi_{s,t}=\varphi_{t}^{-1}\circ R_{(1-\beta(s))s_{0}t+\beta(s)s_{1}t},$$ where $\beta$ increases from $0$ to $1$ over the interval $[0,1]$ with $\beta'(s)\ge 0$.

\subsubsection{Dependence on the system}
\label{sec:dependence-system-1}

If $(\varphi_{1},[\varphi_{t}])=(\varphi_{1}',[\varphi_{t}'])$, then there is an induced isomorphism $V_{\alpha,s}(\varphi_{t})\to V_{\alpha,s}(\varphi_{t}')$ constructed exactly as in \S\ref{sec:dependence-system-1}; one uses a path $\varphi_{s,t}$ interpolating from $\varphi_{0,t}=\varphi_{t}$ and $\varphi_{1,t}=\varphi'_{t}$. We now show these isomorphisms commute with the structure morphisms. Consider the path:
\begin{equation*}
  \xi_{s,t}^{0}=\left\{
    \begin{aligned}
      \varphi_{\beta(s+1),t}^{-1}\circ R_{s_{0} t}\text{ for }s\le 0\\
      \varphi_{1,t}^{-1}\circ R_{s_{0}t+\beta(s)\delta t}\text{ for }s\ge 0,
    \end{aligned}
  \right.
\end{equation*}
This path equals $\xi^{0}_{s,t}=\varphi_{1,t}^{-1}\circ R_{\beta(s)\delta t}\circ \varphi_{1,t}\circ \varphi_{\beta(s+1),t}^{-1}\circ R_{s_{0}t}$, as can be checked directly. Consider the homotopy:
\begin{equation*}
  \xi_{s,t}^{\eta}=\varphi_{1,t}^{-1}\circ R_{(1-\eta)\beta(s+\eta)\delta t}\circ \varphi_{1,t}\circ \varphi_{\beta(s+1-\eta),t}^{-1}\circ R_{s_{0}t+\eta\beta(s+\eta)\delta t},
\end{equation*}
noting that, when $t=1$, this path is $\varphi_{1,1}^{-1}\circ R_{s_{0}+\beta(s+\eta)\delta}$. In particular, $\xi^{\eta}_{s,t}$ is a homotopy through non-negative paths (since $\delta\ge 0$).

Thus the continuation map associated to $\xi^{0}_{s,t}$ equals the continuation map associated to $\xi^{1}_{s,t}=\varphi_{\beta(s),t}^{-1}\circ R_{s_{0}t+\beta(s+1)\delta t}$, which is equal to the piecewise formula:
\begin{equation*}
  \xi_{s,t}^{1}=\left\{
    \begin{aligned}
      &\varphi_{0,t}^{-1}\circ R_{s_{0} t+\beta(s+1)t}\text{ for }s\le 0,\\
      &\varphi_{\beta(s),t}^{-1}\circ R_{s_{0}t+\delta t}\text{ for }s\ge 0.
    \end{aligned}
  \right.
\end{equation*}
Thus the isomorphisms commute with the structure morphisms, and the persistence module $V_{\alpha,s}(\varphi_{t})$ depends only on $(\varphi_{1},[\varphi_{t}])$ up to isomorphism. See \S\ref{sec:interl-two-pers} below for similar arguments with more details.

\subsection{Spectrality of the barcode}
\label{sec:proof-lemma-endpoint-spectrum}

Lemma \ref{lemma:endpoint-spectrum} follows from \cite{ulja_zhang}, and we defer to their paper for the details. Given a path $\varphi_{s,t}\in \Gamma$ so that $\varphi_{s,1}$ is a non-negative path without discriminant points, the idea is to construct systems $\psi_{s,t}$ \emph{with the same fixed points}, whose ideal restrictions are $\varphi_{s,t}$, and then prove that the continuation morphism acts identically on chain level.

\subsection{No half-infinite bars means no full symplectic cohomology}
\label{sec:no-half-infinite-bars}

Let $R_{s}$ be the Reeb flow for time $s$. It is well-known that there is a unit element $1\in \mathrm{HF}(R_{\epsilon t})=V_{\alpha,\epsilon}(\id)$ is mapped to the unit element of $\mathrm{SH}(W)$ under the continuation morphism.

There is a PSS ring homomorphism $H^{*}(W)\to \mathrm{HF}(R_{\epsilon t})$ from the cup product to the pair-of-pants product; see \cite[\S6]{ritter_TQFT}. Here $H^{*}(W)$ should be thought of as the Morse cohomology $\mathrm{HM}(f_{+})$ of a function which is positive and one-homogeneous at infinity.

One can also consider the Morse cohomology $\mathrm{HM}(f_{-})$ of a function $f_{-}$ which is negative and one-homogeneous at infinity; in this case one has an isomorphism $\mathrm{HM}(f_{-})\simeq H^{*}(W,\bd W)$; there is also a PSS morphism $\mathrm{HF}(R_{-\epsilon t})\to \mathrm{HM}(f_{-})$, and it is a well-known folk-theorem that the following diagram commutes:
\begin{equation*}
  \begin{tikzcd}
    {V_{\alpha,-\epsilon}(\id)}\arrow[d,"{}"]\arrow[r,"{\mathfrak{c}}"] &{V_{\alpha,\epsilon}(\id)}\arrow[from=2-2,"{}"]\\
    {\mathrm{HM}(f_{-})}\arrow[r,"{c}"] &{\mathrm{HM}(f_{+})},
  \end{tikzcd}
\end{equation*}
where $\mathfrak{c}$ is the structure morphism defined in \S\ref{sec:defin-pers-module}. For related discussion, see \cite[\S2]{cant_hedicke_kilgore}.

By picking $f_{-}$ to have no local minima, then $c:\mathrm{HM}(f_{-})\to \mathrm{HM}(f_{+})$ obviously misses the unit for index reasons, and hence $\mathfrak{c}$ misses the unit element. Indeed, the image of $c$ consists entirely of nilpotent elements with respect to the cup product. It follows that the image of $\mathrm{HF}(R_{-\epsilon t}^{\alpha})\to \mathrm{SH}(W)$ consists entirely of nilpotent elements.

Finally recall from \cite{ritter_TQFT} that $\mathrm{SH}(W)=0$ if and only if $1\in \mathrm{SH}(W)$ vanishes. Thus, if $\mathrm{SH}(W)\ne 0$, then there is a half-infinite bar corresponding to the unit which is born at $0$.

\subsection{Proof of Theorem \ref{theorem:main}}
\label{sec:proof-1}

As mentioned in \S\ref{sec:stat-main-result}, the idea is to construct an \emph{interleaving} between the persistence modules $V(\varphi_{0})$ and $V(\varphi_{1})$, and then to apply the isometry theorem relating the interleaving distance to the barcode distance. Interleavings are recalled in \S\ref{sec:interleavings-isometry}, and the proof of Theorem \ref{theorem:main} is completed in \S\ref{sec:interl-two-pers}.

\subsubsection{Interleavings and the isometry theorem}
\label{sec:interleavings-isometry}
Let $V_{s},W_{s}$ be two persistence modules. A $\delta$-interleaving is a collection of maps $V_{s}\to W_{s+\delta}$, $W_{s}\to V_{s+\delta}$, so that:
\begin{enumerate}
\item the maps commute with the structure maps, and
\item the compositions $V_{s}\to V_{s+2\delta}$ and $W_{s}\to W_{s+2\delta}$ equal the structure maps.
\end{enumerate}

The \emph{interleaving distance} between $V,W$ is the infimal $\delta$ for which there exists a $\delta$-interleaving; see \cite[\S1.3]{persistence_book}. It is an important theorem of persistence modules and their associated barcodes that:

\begin{theorem}
  The interleaving distance between persistence modules equals the barcode distance between their associated barcodes.
\end{theorem}
\begin{proof}
  See \cite[\S3]{persistence_book}, \cite[\S5.4]{chazal_et_al_structure_stability_pmod_book}, and \cite{cohen_steiner_edelsbrunner_harer,chazal_cohen_steiner_glisse_guibas_oudot_2009,bauer_lesnick} for more details.
\end{proof}

\subsubsection{Interleaving the two persistence modules}
\label{sec:interl-two-pers}
Let $\varphi_{s,t}$ be a path joining systems $\varphi_{0,t}$ and $\varphi_{1,t}$. Abbreviate $\varphi_{s}=\varphi_{s,1}$. Let $\beta(s)$ be as in \S\ref{sec:defin-pers-module}, and pick $\delta$ large enough that $\varphi_{\beta(s)}^{-1}\circ R_{s_{0}+\beta(s)\delta}$ is a positive path; we will momentarily estimate how large $\delta$ needs to be. Use the continuation map defined in \S\ref{sec:defin-cont-map} to obtain a map $V_{\alpha,s_{0}}(\varphi_{0,t})\to V_{\alpha,s_{0}+\delta}(\varphi_{1,t})$. Also choose $\delta$ large enough that $\varphi_{1-\beta(s)}^{-1}\circ R_{s_{0}+\beta(s)\delta}$ is positive, thereby giving maps in the reverse direction $V_{\alpha,s_{0}}(\varphi_{1})\to V_{\alpha,s_{0}+\delta}(\varphi_{0})$.
\begin{lemma}
  If the Shelukhin-Hofer distance between $\varphi_{0},\varphi_{1}$ is smaller than~$\delta$, in the universal cover, then the isotopy $\varphi_{s,t}$ can be chosen so that above paths are positive and the above continuation maps define a $\delta$-interleaving.
\end{lemma}
\begin{proof}
  First we prove the non-negativity of the advertised paths. It suffices to prove the paths $\psi_{s}:=\varphi^{-1}_{s}\circ R_{s_{0}+s\delta}$ and $\varphi^{-1}_{1-s}\circ R_{s_{0}+s\delta}$ are non-negative for $s\in [0,1]$. One computes:
  \begin{equation*}
    \psi_{s}'=[\d\varphi_{s}^{-1}(-X_{s}+\delta R)\circ \varphi_{s}]\circ \psi_{t},
  \end{equation*}
  Inserting the generator of $\psi_{s}$ into the contact form $\varphi_{s}^{*}\alpha$ yields the quantity $(-h_{s}+\delta)\circ \varphi_{s}$; hence a sufficient criterion for positivity is that $\delta>\max_{x,s}\abs{h_{s}}$. By reparametrizing $\varphi_{s}$ in $s$ optimally, we can ensure that $\max_{x,s}\abs{h_{s}}$ is close to $\int \max_{x} \abs{h_{s}}\d s$; see, e.g., \cite[\S3]{nakamura_orderable_metric}. If $\delta$ is larger than $\dist_{\alpha}(\varphi_{0},\varphi_{1})$ then the required paths can be made non-negative (indeed, even positive). The positivity of the other path (using $\varphi_{1-s}$) follows by a symmetry argument, indeed, the contact Hamiltonian for $\varphi_{1-s}$ is $-h_{1-s}$.

  Next we prove the proposed interleaving commutes with the structure morphisms. Standard gluing arguments show the composition of a structure morphism $s_{0}\to s_{1}$ followed by the interleaving morphism is itself the continuation map of the composite path:
  \begin{equation*}
    \psi_{s,t}^{0}=\left\{
      \begin{aligned}
        &\varphi_{0,t}^{-1}\circ R_{(1-\beta(s+1))s_{0}t+\beta(s+1)s_{1}t}\text{ for }s\le 0,\\
        &\varphi_{\beta(s),t}^{-1}\circ R_{s_{1}t+\beta(s)\delta t}\text{ for }s\ge 0.\\
      \end{aligned}
    \right.
  \end{equation*}
  See Figure \ref{fig:cutoff-beta} for the graph of $\beta(s)$.
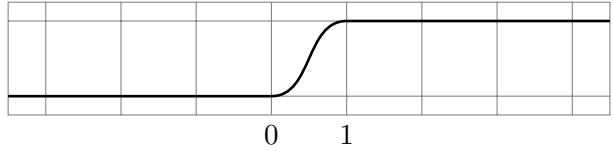
\begin{figure}[H]
    \centering
    \begin{tikzpicture}
      \draw[black!50!white] (-3.5,-0.25)grid(4.5,1.25) (-3.5,-0.25)rectangle(4.5,1.25);
      \draw[line width=1pt] (-3.5,0)--(0,0)to[out=0,in=180](1,1)--+(3.5,0);
      \node at (0,-0.25) [below]{$0$};
      \node at (1,-0.25) [below]{$1$};
    \end{tikzpicture}
    \caption{Cut-off function $\beta(s)$.}
    \label{fig:cutoff-beta}
  \end{figure}
  
  A bit of thought shows that $\psi^{0}_{s,t}$ is equal to:
  \begin{equation}\label{eq:interleaving_A}
    \psi_{s,t}^{0}=\varphi_{\beta(s),t}^{-1}R_{s_{0}t+\beta(s)\delta t +\beta(s+1)(s_{1}-s_{0})t}.
  \end{equation}
  The morphism obtained by composing the morphisms in the opposite order is the continuation morphism associated to the path:
  \begin{equation*}
    \psi_{s,t}^{1}=\left\{
      \begin{aligned}
        &\varphi_{\beta(s+1),t}^{-1}\circ R_{s_{0}t+\beta(s+1)\delta t}\text{ for }s\le 0,\\
        &\varphi_{1,t}^{-1}\circ R_{(1-\beta(s))s_{0}t+\beta(s)s_{1}t+\delta t}\text{ for }s\ge 0,\\
      \end{aligned}
    \right.
  \end{equation*}
  and this equals:
  \begin{equation}\label{eq:interleaving_B}
    \psi_{s,t}^{1}=\varphi_{1,t}^{-1}\circ R_{\beta(s)(s_{1}-s_{0})t}\circ \varphi_{1,t}\circ \varphi_{\beta(s+1),t}^{-1}\circ R_{s_{0}t+\beta(s+1)\delta t}.
  \end{equation}
  As explained in \S\ref{sec:invar-under-homot}, the two compositions are equal provided we can homotope \eqref{eq:interleaving_A} to \eqref{eq:interleaving_B} through non-negative paths with fixed endpoints (bearing in mind that the endpoints are $s=-\infty$ and $s=+\infty$).

  Consider the family of paths $\psi^{\eta}_{s}$, parameterized by $\eta \in [0,1]$, given by:
  \begin{equation*}
    \underbrace{\varphi_{1}^{-1}\circ R_{\eta\beta(s+1-\eta)(s_{1}-s_{0})}\circ \varphi_{1}}_{\text{non-negative}}\circ \underbrace{\varphi_{\beta(s+\eta)}^{-1}\circ R_{s_{0}+\beta(s+\eta)\delta+(1-\eta)\beta(s+1-\eta)(s_{1}-s_{0})}}_{\text{non-negative}},
  \end{equation*}
  where we suppress the dependence on $t$ in the notation (i.e., the notation sets $t=1$).
  
  When $\eta=0$ this path equals \eqref{eq:interleaving_A}, while when $\eta=1$ it equals \eqref{eq:interleaving_B}. The endpoints of the path remain fixed at $\varphi_{0}^{-1}\circ R_{s_{0}}$ and $\varphi_{1}^{-1}\circ R_{s_{1}+\delta}$. Moreover, the non-negativity of the two paths ensures that $\psi_{s}^{\eta}$ is a deformation through non-negative paths.

  Thus the proposed interleaving indeed commutes with the structure maps.

  Next we prove that the composition of the two interleaving maps is the structure map $V_{s_{0}}\to V_{s_{0}+2\delta}$. The argument is similar to the above, and we will again suppress the dependence on $t$ from the notation by setting $t=1$. The composition of the two interleavings is given by the piecewise formula:
  \begin{equation*}
    \psi_{s}^{1}=\left\{
      \begin{aligned}
        &\varphi_{\beta(s+1)}^{-1}\circ R_{s_{0}+\beta(s+1)\delta}\text{ for }s\le 0.\\
        &\varphi_{1-\beta(s)}^{-1}\circ R_{s_{0}+\delta+\beta(s)\delta}\text{ for }s\ge 0,\\
      \end{aligned}
    \right.
  \end{equation*}
  Similarly to the above, this is given by the formula:
  \begin{equation*}
    \psi_{s}^{1}=\varphi_{1-\beta(s)}^{-1}\circ R_{\beta(s)\delta}\circ \varphi_{1}\circ \varphi_{\beta(s+1)}^{-1}\circ R_{s_{0}+\beta(s+1)\delta}.
  \end{equation*}
  Consider the deformation:
  \begin{equation*}
    \psi^{\eta}_{s}=\underbrace{\varphi_{\eta(1-\beta(s))}^{-1}\circ R_{\beta(s)\delta}}_{\text{non-negative}}\circ\varphi_{\eta}\circ \underbrace{\varphi_{\eta\beta(s+\eta)}^{-1}\circ R_{s_{0}+\beta(s+\eta)\delta}}_{\text{non-negative}}.
  \end{equation*}
  where $\eta=1$ we recover $\psi^{1}_{s}$, while when $\eta=0$ we recover $\varphi^{-1}_{0}\circ R_{s_{0}+\beta(s)2\delta}$ whose associated continuation map is the structure map, by definition. It is clear that this path remains non-negative, and has fixed endpoints. The desired result follows from the invariance of the continuation map under deformations, as above.
\end{proof}
Applying the interleaving isometry theorem from \S\ref{sec:interleavings-isometry} completes the proof of Theorem \ref{theorem:main}.

\subsubsection{Lipschitz continuity of spectral invariants}
\label{sec:lipsh-cont-spectr-proof}
The Lipschitz continuity of $c_{\alpha}(\varphi,\mathfrak{e})$ with respect to $\dist_{\alpha}$ (with $\alpha$ fixed) follows from \S\ref{sec:interl-two-pers}. Indeed, for any $\delta>\dist_{\alpha}(\varphi_{0},\varphi_{1})$, we have constructed a sequence:
\begin{equation*}
  \begin{tikzcd}
    {V_{\alpha,s-\delta}(\varphi_{1,t})}\arrow[r,"{}"]&{V_{\alpha,s}(\varphi_{0,t})}\arrow[r,"{}"]&{V_{\alpha,s+\delta}(\varphi_{1,t})}\arrow[r,"{}"]&{\mathrm{SH}(W)},
  \end{tikzcd}
\end{equation*}
using continuation maps. Picking $s$ nearby $s_{0}=c_{\alpha}(\varphi_{0},\mathfrak{e})$, both slightly larger and slightly smaller, and considering when $\mathfrak{e}$ lies in the image of the map to $\mathrm{SH}$ yields:
\begin{equation*}
  c_{\alpha}(\varphi_{0},\mathfrak{e})-\delta  \le c_{\alpha}(\varphi_{1},\mathfrak{e})\le c_{\alpha}(\varphi_{0},\mathfrak{e})+\delta.
\end{equation*}
Taking $\delta\to \dist_{\alpha}(\varphi_{0},\varphi_{1})$ completes the first part of Proposition \ref{prop:lipschitz-spectral}.

The continuity with respect to variations in the contact form can be proved as follows: for any $\delta>0$ and any pair $(s_{0},\alpha)$, one can find a $C^{1}$ neighborhood $U$ of $\alpha$ and an open interval $I$ around $s_{0}$ so that the following holds for all $\beta\in U$ and $s\in I$:
\begin{enumerate}
\item $\dist_{\beta}(\varphi_{1},\varphi_{2})<2\dist_{\alpha}(\varphi_{1},\varphi_{2})$, and,
\item There exists a positive path from $R_{s-\delta}^{\beta}$ to $R_{s}^{\alpha}$ to $R^{\beta}_{s+\delta}$.
\end{enumerate}
For fixed $\varphi_{0,t}$, (ii) yields a diagram of continuation morphisms:
\begin{equation*}
  \begin{tikzcd}    {V_{\beta,s-\delta}(\varphi_{0,t})}\arrow[r,"{}"]&{V_{\alpha,s}(\varphi_{0,t})}\arrow[r,"{}"]&{V_{\beta,s+\delta}(\varphi_{0,t})}\arrow[r,"{}"]&{\mathrm{SH}(W)}.
  \end{tikzcd}
\end{equation*}
Pick $s_{0}=c_{\alpha}(\varphi_{0},\mathfrak{e})$, then by taking $s<s_{0}$ we conclude that the composite of all three maps does not have $\mathfrak{e}$ in its image, which implies $c_{\beta}(\varphi_{0},\mathfrak{e})\ge s-\delta$. Then take $s\to s_{0}$ to conclude $c_{\beta}(\varphi_{0},\mathfrak{e})\ge c_{\alpha}(\varphi_{0},\mathfrak{e})-\delta$. Similarly take $s>s_{0}$ to conclude: $$c_{\beta}(\varphi_{0},\mathfrak{e})\le c_{\alpha}(\varphi_{0})+\delta.$$

Finally, if $\varphi_{1}$ is nearby $\varphi_{0}$, use the Lipschitz-continuity to estimate for $\beta\in U$ that:
\begin{equation*}
  \abs{c_{\beta}(\varphi_{1},\mathfrak{e})-c_{\alpha}(\varphi_{0},\mathfrak{e})}\le 2\dist_{\alpha}(\varphi_{0},\varphi_{1})+\delta.
\end{equation*}
The second part of Proposition \ref{prop:lipschitz-spectral} follows. Moreover we have proved that $c_{\alpha}(\varphi,\mathfrak{e})$ is continuous as a function of two variables $(\alpha,\varphi)$, when $\alpha$ is endowed with the $C^{1}$ topology and $\varphi$ with the topology induced by $\dist_{\alpha}$.

\subsubsection{Monotonicity of spectral invariants}
\label{sec:mono-spectr-proof}

Suppose that there is a path $\varphi_{s,t}$ with $\varphi_{0,1},\varphi_{1,1}\in \Gamma^{\times}$ so $\varphi_{s,1}$ is positive. Then the path $\varphi_{1-s}^{-1}$ is positive, and hence one can find a small $\delta>0$ so that the path $\varphi_{1-s}^{-1}\circ R_{s_{0}+(1-s)\delta}$ is positive. Therefore the following diagram commutes, where the morphisms are continuation maps:
\begin{equation*}
  \begin{tikzcd}
    {V_{\alpha,s_{0}+\delta}(\varphi_{1,t})}\arrow[r,"{}"] &{V_{\alpha,s_{0}}(\varphi_{0,t})}\arrow[r,"{}"]&{\mathrm{SH}(W)},
  \end{tikzcd}
\end{equation*}
Take $s_{0}=c_{\alpha}(\varphi_{0},\mathfrak{e})-\epsilon$, so that the second map does \emph{not} have $\mathfrak{e}$ in its image. It follows that the composite map also does not have $\mathfrak{e}$ in its image, and hence: $$c_{\alpha}(\varphi_{1},\mathfrak{e})>c_{\alpha}(\varphi_{0},\mathfrak{e})-\epsilon+\delta;$$
by taking $\epsilon<\delta$, we conclude Proposition \ref{prop:mono-spectral}.

\subsubsection{Bounding the boundary depth}
\label{sec:bound-bound-depth}

We prove Proposition \ref{prop:boundary-depth-estimate}. First, let $B$ be some number so that $s\mapsto R_{Bs}\circ \zeta_{s,t}$ is positive for each $t$. Differentiating the path with respect to $s$ shows that the condition:
\begin{equation*}
  B+\min_{s,t,z}h_{s,t}(z)>0
\end{equation*}
is necessary and sufficient. Thus one can take $B$ slightly larger than the measurement $\mu_{\alpha}(\zeta_{s,t})$ introduced in \S\ref{sec:order-ideal-bound}. The goal is to prove that $b\le B$, where $b$ is the boundary depth of $\mathfrak{B}_{\alpha}(\varphi)$.

Pick $\epsilon>0$ small enough so that $s\mapsto R_{-\epsilon s}\circ \zeta_{s,1}$ is positive (this is possible since $s\mapsto \zeta_{s,1}$ is positive). Similarly to the arguments in \S\ref{sec:interl-two-pers}, consider the positive path:
\begin{equation*}
  \xi_{s,t}:=\left\{
    \begin{aligned}
      &\varphi^{-1}_{t}\circ R_{s_{0}t}\circ R_{-\epsilon \beta(s+1)t}\circ \zeta_{\beta(s+1),t}\text{ for }s\le 0,\\
      &\varphi^{-1}_{t}\circ R_{s_{0}t-\epsilon t+(B+\epsilon)\beta(s)t}\text{ for }s\ge 0.
    \end{aligned}
  \right.
\end{equation*}
Gluing analysis implies that the associated continuation map is the composition of the continuation map for the first path with the continuation map for the second path. Observe that:
\begin{equation*}
  \xi_{s,t}=\varphi_{t}^{-1}\circ R_{s_{0}t+(B+\epsilon)\beta(s)t-\epsilon\beta(s+1)t}\circ \zeta_{\beta(s+1),t}.
\end{equation*}
By interpolating $s+1$ to $s$, one sees $\xi_{s,t}$ is homotopic to the modified path: $$\xi_{s,t}=\varphi_{t}^{-1}\circ R_{s_{0}t+B\beta(s)t}\circ \zeta_{\beta(s),t},$$ with fixed endpoints, so that $\xi_{s,1}$ remains positive. Finally, one considers the homotopy:
\begin{equation*}
  \xi_{s,t}^{\eta}=\varphi_{t}^{-1}\circ R_{s_{0}t+B\beta(s)t}\circ \zeta_{\beta(s),\eta t},
\end{equation*}
where $\eta$ goes from $1$ to $0$; by assumption on $B$, $\xi_{s,1}^{\eta}$ remains a positive path, and this homotopy terminates at $\xi_{s,t}^{0}=\varphi_{t}^{-1}\circ R_{s_{0}t+B\beta(s)t}$. By the invariance of the continuation morphism under homotopies through non-negative paths, the following diagram commutes:
\begin{equation}\label{eq:diagram-commutes-boundary-depth}
  \begin{tikzcd}
    {V_{s_{0}}^{\alpha}(\varphi_{t})}\arrow[rd,swap,"{\mathfrak{c}_{s_{0},s_{0}+B}}"]\arrow[r,"{\mathfrak{c}'}"] &{V_{s_{0}-\epsilon}^{\alpha}(\varphi_{t})}\arrow[d,"{\mathfrak{c}_{s_{0}-\epsilon,s_{0}+B}}"]\\
    &{V_{s_{0}+B}^{\alpha}(\varphi_{t})},
  \end{tikzcd}
\end{equation}
where $\mathfrak{c}'$ is the continuation map associated to $\xi_{s,t}$ restricted to $s\in [-1,0]$. The diagonal morphism is generated by $\xi^{0}_{s,t}$ and the composition of the vertical and horizontal morphisms is generated by $\xi_{s,t}^{1}$.

If $B<b$, then one can find some $s_{0}$ so that the image of $\mathfrak{c}_{s_{0},s_{0}+B}$ is strictly larger than the image of $\mathfrak{c}_{s_{0}-\epsilon,s_{0}+B}$; one can pick some bar close to maximal length and take $s_{0}$ to be slightly to the right of the start of this bar, so $s_{0}-\epsilon$ lies to the left of the bar, and so $s_{0}+B$ is still inside the bar; see Figure \ref{fig:boundary-depth-trick}.

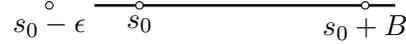
\begin{figure}[H]
  \centering
  \begin{tikzpicture}[xscale=2]
    \draw[line width=1pt] (0,0)--(2,0);
    \path[every node/.style={draw,circle,fill=white,inner sep=1pt}] (-0.3,0)node{}--(0.3,0)node{}--(1.8,0)node{};
    \path[every node/.style={below}] (-0.3,0)node{$s_{0}-\epsilon$}--(0.3,0)node{$s_{0}$}--(1.8,0)node{$s_{0}+B$};
  \end{tikzpicture}
  \caption{Choosing $s_0$ if $B$ is smaller than the length of the longest bar.}
  \label{fig:boundary-depth-trick}
\end{figure}

This contradicts the commutative diagram \eqref{eq:diagram-commutes-boundary-depth}, and hence $B\ge b$, completing the proof of Proposition \ref{prop:boundary-depth-estimate}.

\bibliographystyle{alpha}
\bibliography{citations}

\newcommand{\etalchar}[1]{$^{#1}$}
\begin{thebibliography}{CCSG{\etalchar{+}}09}

\bibitem[AA23]{allais_arlove}
S.~Allais and P-A. Arlove.
\newblock Spectral selectors and contact orderability.
\newblock arXiv:2309.10578, 2023.

\bibitem[AAC23]{alizadeh-atallah-cant}
H.~Alizadeh, M.~S. Atallah, and D.~Cant.
\newblock Lagrangian intersections and the spectral norm in convex-at-infinity
  symplectic manifolds.
\newblock arXiv:2312.14752, December 2023.

\bibitem[Abo15]{abouzaid_monograph}
M.~Abouzaid.
\newblock Symplectic cohomology and {V}iterbo's theorem.
\newblock In {\em Free Loop Spaces in Geometry and Topology}, pages 271--486.
  European Mathematical Society, 2015.

\bibitem[AF12]{albers_frauenfelder_nonlinear_maslov}
P.~Albers and U.~Frauenfelder.
\newblock A variational approach to {G}ivental's nonlinear {M}aslov index.
\newblock {\em Geom. Funct. Anal.}, 22:1033--1050, 2012.

\bibitem[AFM15]{AFM15}
P.~Albers, U.~Fuchs, and W.~J. Merry.
\newblock Orderability and the {W}einstein conjecture.
\newblock {\em Compositio Mathematica}, 151(12):2251–2272, 2015.

\bibitem[All22]{allais_zoll}
S.~Allais.
\newblock {M}orse estimates for translated points on unit tangent bundles.
\newblock arXiv, 2022.

\bibitem[Alu09]{aluffi}
P.~Aluffi.
\newblock {\em Algebra: Chapter 0}, volume 104 of {\em Graduate Studies in
  Mathematics}.
\newblock AMS, 2009.

\bibitem[AM13]{albers_merry}
P.~Albers and W.~J. Merry.
\newblock Translated points and {R}abinowitz {F}loer homology.
\newblock {\em Journal of Fixed Point Theory and Applications}, 13:201--214,
  2013.

\bibitem[AM18]{albers_merry_orderability_non_squeezing}
P.~Albers and W.~J. Merry.
\newblock Orderability, contact non-squeezing, and {R}abinowitz {F}loer
  homology.
\newblock {\em J. Symp. Geom.}, 16(6):1481--1547, 2018.

\bibitem[ASZ16]{albers_shelukhin_zapolsky}
P.~Albers, E.~Shelukhin, and F.~Zapolsky.
\newblock Spectral invariants for contactomorphisms of prequantization bundles
  and applications.
\newblock In preparation; {\url{https://youtu.be/DWel-3BOQrl}}, 2016.

\bibitem[BC23]{brocic_cant}
F.~Bro\'ci\'c and D.~Cant.
\newblock Bordism classes of loops and {F}loer's equation in cotangent bundles.
\newblock arXiv:2305.11783v1, 2023.

\bibitem[BEH{\etalchar{+}}03]{BEHWZ}
F.~Bourgeois, Y.~Eliashberg, H.~Hofer, K.~Wysocki, and E.~Zehnder.
\newblock Compactness results in {Symplectic Field Theory}.
\newblock {\em Geometry and Topology}, 7:799--888, 2003.

\bibitem[BK22]{benedetti-kang}
G~Benedetti and J.~Kang.
\newblock Relative {H}ofer-{Z}ehnder capacity and positive symplectic homology.
\newblock {\em J. Fixed Point Theory Appl.}, 24(44):1--32, 2022.

\bibitem[BL15]{bauer_lesnick}
U.~Bauer and M.~Lesnick.
\newblock Induced matchings and the algebraic stability of persistence modules.
\newblock {\em J. Comput. Geom.}, 6(2):162--191, 2015.

\bibitem[Can22a]{cant_sandon_conj}
D.~Cant.
\newblock Contactomorphisms of the sphere without translated points.
\newblock arXiv:2210.11002, 2022.

\bibitem[Can22b]{cant_thesis}
D.~Cant.
\newblock A dimension formula for relative symplectic field theory.
\newblock Stanford University PhD Thesis. Available at:
  {\url{https://dylancant.ca/thesis.pdf}}, 2022.

\bibitem[Can23]{cant_oscillation_energy}
D.~Cant.
\newblock Remarks on the oscillation energy of {L}egendrian isotopies.
\newblock arXiv:2301.06205, 2023.

\bibitem[CB15]{crawley_boevey}
W.~Crawley-Boevey.
\newblock Decomposition of pointwise finite-dimensional persistence modules.
\newblock {\em J. Algebra Appl.}, 14(5):8, 2015.

\bibitem[CCDR19]{chantraine_colin_d_rizell}
B.~Chantraine, V.~Colin, and G.~Dimitroglou~Rizell.
\newblock Positive {L}egendrian isotopies and {F}loer theory.
\newblock {\em Ann. Inst. Fourier}, 69(4):1679--1737, 2019.

\bibitem[CCSG{\etalchar{+}}09]{chazal_cohen_steiner_glisse_guibas_oudot_2009}
F.~Chazal, D.~Cohen-Steiner, M.~Glisse, L.~J. Guibas, and S.~Y. Oudot.
\newblock Proximity of persistence modules and their diagrams.
\newblock In {\em Proceedings of the Twenty-Fifth Annual Symposium on
  Computational Geometry}, SCG '09, page 237–246, 2009.

\bibitem[CdSGO16]{chazal_et_al_structure_stability_pmod_book}
F.~Chazal, V.~de~Silva, M.~Glisse, and S.~Oudot.
\newblock {\em The Structure and Stability of Persistence Modules}.
\newblock Springer, 2016.

\bibitem[CE12]{cieliebak_eliashberg_stein}
K.~Cieliebak and Y.~Eliashberg.
\newblock {\em From {S}tein to {W}einstein and Back; Symplectic Geometry of
  Affine Complex Manifolds}, volume~59 of {\em Colloquium Publications}.
\newblock AMS, 2012.

\bibitem[CF09]{cieliebak_frauenfelder}
K.~Cieliebak and U.~Frauenfelder.
\newblock A {F}loer homology for exact contact embeddings.
\newblock {\em Pacific J. Math.}, 239(2):251--316, 2009.

\bibitem[CFO10]{cieliebak_frauenfelder_oancea}
K.~Cieliebak, U.~Frauenfelder, and A.~Oancea.
\newblock Rabinowitz floer homology and symplectic homology.
\newblock {\em Ann. Scient. {\'E}c. Norm. Sup.}, 4(43):957--1015, 2010.

\bibitem[CGG22]{cineli_ginzburg_gurel_entropy}
E.~Cineli, V.~L. Ginzburg, and B.~Z. G{\"{u}}rel.
\newblock Topological entropy of {H}amiltonian diffeomorphisms: a persistence
  homology and {F}loer theory perspective.
\newblock arXiv:2111.03983, 2022.

\bibitem[Che96]{chekanovQF}
Yu.~V. Chekanov.
\newblock Critical points of quasi-functions and generating families of
  {Legendrian} manifolds.
\newblock {\em Functional Analysis and its Applications}, 30(2):118--128, 1996.

\bibitem[CHK23]{cant_hedicke_kilgore}
D.~Cant, J.~Hedicke, and E.~Kilgore.
\newblock Extensible positive loops and vanishing of symplectic cohomology.
\newblock arXiv:2311.18267, 2023.

\bibitem[Cie02]{cieliebak_handle_chord}
K.~Cieliebak.
\newblock Handle attaching in symplectic homology and the chord conjecture.
\newblock {\em J. European Math. Soc.}, 4(2):115--142, 2002.

\bibitem[CS15]{colin_sandon}
V.~Colin and S.~Sandon.
\newblock The discriminant and oscillation lengths for contact and legendrian
  isotopies.
\newblock {\em J. Eur. Math. Soc.}, 17:1657--1685, 2015.

\bibitem[CSEH07]{cohen_steiner_edelsbrunner_harer}
D.~Cohen-Steiner, H.~Edelsbrunner, and J.~Harer.
\newblock Stability of persistence diagrams.
\newblock {\em Disc. Comput. Geom.}, 37:103--120, 2007.

\bibitem[DRS20]{rizell_sullivan_persistence_1}
G.~Dimitroglou~Rizell and M.~G. Sullivan.
\newblock The persistence of the {Chekanov-Eliashberg} algebra.
\newblock {\em Sel. Math.}, 26(69):32, 2020.

\bibitem[DRS21]{rizell_sullivan_augmentation}
G.~Dimitroglou~Rizell and M.~G. Sullivan.
\newblock The persistence of a relative {R}abinowitz-{F}loer complex.
\newblock arXiv, 2021.

\bibitem[DS93]{dostoglou_salamon}
S.~Dostoglou and D.~Salamon.
\newblock Instanton homology and symplectic fixed points.
\newblock {\em London Math. Soc. Lecture Note Ser.}, 192:53--93, 1993.

\bibitem[DU22]{ulja_drobnjak}
D.~Drobnjak and I.~Uljarevi{\'c}.
\newblock Exotic symplectomorphisms and contact circle actions.
\newblock {\em Comm. Contemp. Math.}, 23(03), 2022.

\bibitem[DUZ23]{djordjevic_uljarevic_zhang}
D.~Djordjevi{\'c}, I.~Uljarevi{\'c}, and J.~Zhang.
\newblock Quantitative characterization in contact {H}amiltonian dynamics -
  {I}.
\newblock arXiv:2309.00527, 2023.

\bibitem[EGH00]{introSFT}
Y.~Eliashberg, A.~Givental, and H.~Hofer.
\newblock Introduction to symplectic field theory.
\newblock In {\em Visions in Mathematics: GAFA 2000 Special volume, Part II},
  pages 560--673. Birkh{\"a}user Basel, 2000.

\bibitem[EKP06]{ekp}
Y.~Eliashberg, S.~S. Kim, and L.~Polterovich.
\newblock Geometry of contact transformations and domains: orderability versus
  squeezing.
\newblock {\em Geom. Topol.}, 10(3):1635--1747, 2006.

\bibitem[EP00]{ep2000}
Y.~Eliashberg and L.~Polterovich.
\newblock Partially ordered groups and geometry of contact transformations.
\newblock {\em Geom. Funct. Anal.}, 10:1448--1476, 2000.

\bibitem[Flo89]{floer-ham}
A.~Floer.
\newblock Symplectic fixed points and holomorphic spheres.
\newblock {\em Communications in Mathematical Physics}, 120:575--611, 1989.

\bibitem[FPR18]{fraser_polterovich_rosen}
M.~Fraser, L.~Polterovich, and D.~Rosen.
\newblock On {S}andon-type metrics for contactomorphism groups.
\newblock {\em Annales math\'ematiques du Qu\'ebec}, 42:191--214, 2018.

\bibitem[FSB23]{fender-lee-sohn-entropy}
E.~Fender, L.~Sangjin, and S.~Beomjun.
\newblock Barcode entropy for {R}eeb flows on contact manifolds with
  {L}iouville fillings.
\newblock arXiv:2305.04770, 2023.

\bibitem[GGM22]{ginzburg_gurel_mazzucchelli_entropy}
V.~L. Ginzburg, B.~Z. G{\"{u}}rel, and M~Mazzucchelli.
\newblock Barcode entropy of geodesic flows.
\newblock arXiv:2212.00943, 2022.

\bibitem[Gir17]{cool_gadget}
E.~Giroux.
\newblock Ideal {L}iouville domains; a cool gadget.
\newblock arXiv:1708.08855, 2017.

\bibitem[Giv90]{givental_quasimorphism}
A.~B. Givental.
\newblock Nonlinear generalization of the {M}aslov index.
\newblock In {\em Theory of singularities and its applications}, Advances in
  Soviet Mathematics, pages 71--103. American Mathematical Society, 1990.

\bibitem[GKPS17]{gkps}
G.~Granja, Y.~Karshon, M.~Pabiniak, and S.~Sandon.
\newblock {G}ivental's non-linear {M}aslov index on {L}ens spaces.
\newblock arXiv, 2017.

\bibitem[KS21]{kislev_shelukhin}
A.~Kislev and E.~Shelukhin.
\newblock Bounds on spectral norms and barcodes.
\newblock {\em Geom. Topol.}, 25:3257--3350, 2021.

\bibitem[Mai22]{PA_spectral_diameter}
P-A. Mailhot.
\newblock The spectral diameter of a {L}iouville domain.
\newblock arXiv:2204.04618, 2022.

\bibitem[MK18]{meiwes_naef}
M.~Matthias and N.~Kathrin.
\newblock Translated points on hypertight contact manifolds.
\newblock {\em Journal of Topology and Analysis}, 2018.

\bibitem[MS12]{mcduffsalamon}
D.~McDuff and D.~Salamon.
\newblock {\em $J$-holomorphic curves and Symplectic Topology}.
\newblock American Mathematical Society, Colloquium Publications, 2nd edition,
  2012.

\bibitem[MU19]{merry_ulja}
W.~J. Merry and I.~Uljarevi\'c.
\newblock Maximum principles in symplectic homology.
\newblock {\em Israel Journal of Mathematics}, 229:39--65, 2019.

\bibitem[Nak23]{nakamura_orderable_metric}
L.~Nakamura.
\newblock a new metric on the contactomorphism group of orderable contact
  manifolds.
\newblock arXiv:2307:10905, 2023.

\bibitem[Oh21]{oh_legendrian_entanglement}
Y-G. Oh.
\newblock Geometry and analysis of contact instantons and entanglement of
  {L}egendrian links {I}.
\newblock arXiv:2111.02597, 2021.

\bibitem[Oh22]{oh_shelukhin_anti_ci}
Y-G. Oh.
\newblock Contact instantons, anti-contact involution and proof of
  {S}helukhin's conjecture.
\newblock arXiv:2212.03557, 2022.

\bibitem[OY23a]{oh_yu_1}
Y-G. Oh and S.~Yu.
\newblock Contact instantons with {L}egendrian boundary condition: a priori
  estimates, asymptotic convergence and index formula.
\newblock arXiv:2301.06023, 2023.

\bibitem[OY23b]{oh_yu_2}
Y-G. Oh and S.~Yu.
\newblock Legendrian contact instantons cohomology and its spectral invariants
  on the one-jet bundle.
\newblock arXiv:2301.06704, 2023.

\bibitem[PRSZ20]{persistence_book}
L.~Polterovich, D.~Rosen, K.~Samvelyan, and J.~Zhang.
\newblock {\em Topological persistence in geometry and analysis}, volume~74.
\newblock AMS, 2020.

\bibitem[PS16]{polterovich_shelukhin_persistence_1}
L.~Polterovich and E.~Shelukhin.
\newblock Autonomous {H}amiltonian flows, {H}ofer's geometry, and persistence
  modules.
\newblock {\em Sel. Math. New. Ser.}, 22:227--296, 2016.

\bibitem[Rit13]{ritter_TQFT}
A.~Ritter.
\newblock Topological quantum field theory structure on symplectic cohomology.
\newblock {\em J. Topol.}, 6:391--489, 2013.

\bibitem[RZ21]{rosen_zhang_CBM}
D.~Rosen and J.~Zhang.
\newblock Relative growth rate and contact {B}anach-{M}azur distance.
\newblock {\em Geom. Dedicata}, 215:1--30, 2021.

\bibitem[San10]{sandon_integer_metric}
S.~Sandon.
\newblock An integer valued bi-invariant metric on the group of
  contactomorphisms of {$\R^{2n}\times S^1$}.
\newblock {\em J. Topol. Anal.}, 2(3):327--339, 2010.

\bibitem[San11]{sandon_non_squeezing}
S.~Sandon.
\newblock Contact homology, capacity and non-squeezing in {$\R^{2n}\times S^1$}
  via generating functions.
\newblock {\em Ann. Inst. Fourier}, 61(1):145--185, 2011.

\bibitem[San12]{sandon_12}
S.~Sandon.
\newblock On iterated translated points for contactomorphisms of
  $\mathbb{R}^{2n+1}$ and $\mathbb{R}^{2n}\times \mathrm{S}^1$.
\newblock {\em International Journal of Mathematics}, 23(2), 2012.

\bibitem[San13]{sandon_13}
S.~Sandon.
\newblock A {M}orse estimate for translated points of contactomorphisms of
  spheres and projective spaces.
\newblock {\em Geom. Dedicata}, 165:95--110, 2013.

\bibitem[San15]{sandon_bi_invariant}
S.~Sandon.
\newblock Bi-invariant metrics on the contactomorphism groups.
\newblock {\em {S\~{a}o} Paulo J. Math. Sci.}, 9:195--228, 2015.

\bibitem[Sei97]{seidel_representation}
P.~Seidel.
\newblock $\pi_1$ of symplectic automorphism groups and invertibles in quantum
  homology rings.
\newblock {\em Geom. funct. anal.}, 7:1046--1095, 1997.

\bibitem[Sei08]{seidel-biased}
P.~Seidel.
\newblock A biased view of symplectic cohomology.
\newblock In {\em Current developments in mathematics, 2006}, pages 211--253.
  Int. Press, Somerville, MA, 2008.

\bibitem[She17]{shelukhin_contactomorphism}
E.~Shelukhin.
\newblock The {H}ofer norm of a contactomorphism.
\newblock {\em J. Symp. Geom.}, 15, 11 2017.

\bibitem[She22a]{shelukhin_viterbo}
E.~Shelukhin.
\newblock Symplectic cohomology and a conjecture of {V}iterbo.
\newblock {\em Geom. Funct. Anal.}, 2022.

\bibitem[She22b]{shelukhin_zoll}
E.~Shelukhin.
\newblock {V}iterbo conjecture for {Z}oll symmetric spaces.
\newblock {\em Invent. math.}, 230:321--373, 2022.

\bibitem[SZ92]{salamon-zehnder}
D.~Salamon and E.~Zehnder.
\newblock Morse theory for periodic solutions of {H}amiltonian systems and the
  {M}aslov index.
\newblock {\em Comm. Pure. Appl. Math.}, 45:1303–1360, 1992.

\bibitem[SZ19]{vukasin_zhang}
V.~Stojisavljevi\'c and J.~Zhang.
\newblock Persistence modules, symplectic {B}anach-{M}azur distance, and
  riemannian metrics.
\newblock arXiv:1810.11151, 2019.

\bibitem[Ulj17]{uljarevic_floer_homology_domains}
I.~Uljarevi{\'c}.
\newblock Floer homology of automorphisms of {L}iouville domains.
\newblock {\em J. Symp. Geom.}, 15(3):861--903, 2017.

\bibitem[Ulj22]{uljarevic_ssh}
I.~Uljarevi\'c.
\newblock Selective symplectic homology with applications to contact
  non-squeezing, 2022.

\bibitem[Ush22]{usher_banach_mazur}
M.~Usher.
\newblock Symplectic {B}anach-{M}azur distance between subsets of {$\C^{n}$}.
\newblock {\em J. Topol. Anal.}, 14(1):231--286, 2022.

\bibitem[UZ16]{usher_zhang}
M.~Usher and J.~Zhang.
\newblock Persistent homology and {Floer-Novikov} theory.
\newblock {\em Geom. Topol.}, 20(6):3333--3430, 2016.

\bibitem[UZ22]{ulja_zhang}
I.~Uljarevi\'c and J.~Zhang.
\newblock Hamiltonian perturbations in contact {F}loer homology.
\newblock arXiv:2203.12300, 2022.

\bibitem[Vit99]{viterbo_functors_and_computations_1}
C.~Viterbo.
\newblock Functors and computations in {F}loer homology with applications, i.
\newblock {\em Geom. Funct. Anal.}, 9(5):985--1033, 1999.

\bibitem[Wen20]{wendl-sft}
C.~Wendl.
\newblock Lectures on symplectic field theory.
\newblock
  \url{https://www.mathematik.hu-berlin.de/~wendl/Sommer2020/SFT/lecturenotes.pdf},
  2020.

\bibitem[Zap13]{zapolsky_cont_groups}
F.~Zapolsky.
\newblock Geometry of contactomorphism groups, contact rigidity, and contact
  dynamics in jet spaces.
\newblock {\em Int. Math. Res. Not.}, 2013(20):4687--4711, 2013.

\end{thebibliography}
\end{document}